\documentclass[preprint,12pt]{elsarticle}

\usepackage{amssymb}
\usepackage{amsmath, amsfonts, vmargin, enumerate}
\usepackage{verbatim}
\usepackage{amsthm}

\newtheorem{theorem}{Theorem}[section]
\newtheorem{corollary}{Corollary}[section]
\newtheorem{definition}{Definition}[section]
\newtheorem{lemma}{Lemma}[section]
\newtheorem{proposition}{Proposition}[section]
\newtheorem{remark}{Remark}[section]

\numberwithin{equation}{section}
\newcommand{\real}{{\mathbb R}}

\newcommand{\B}{{\mathcal B}}
\newcommand{\E}{{\mathcal E}}
\newcommand{\F}{{\mathcal F}}

\newcommand{\M}{{\mathcal M}}
\newcommand{\N}{{\mathcal N}}
\newcommand{\R}{{\mathcal R}}
\newcommand{\U}{{\mathcal U}}
\renewcommand{\P}{{\mathcal P}}
\renewcommand{\H}{{\mathcal H}}
\newcommand{\h}{\mathsf{h}}

\newcommand{\Tr}{\mbox{\rm Tr}}
\newcommand{\tr}{\mbox{\rm tr}}

\newcommand{\8}{\infty}
\newcommand{\el}{\ell}

\newcommand{\be}{\begin{eqnarray*}}
\newcommand{\ee}{\end{eqnarray*}}
\newcommand{\beq}{\begin{equation}}
\newcommand{\eeq}{\end{equation}}
\newcommand{\beqn}{\begin{equation*}}
\newcommand{\eeqn}{\end{equation*}}
\newcommand{\bsp}{\begin{split}}
\newcommand{\esp}{\end{split}}

\begin{document}

\begin{frontmatter}



\title{Haagerup noncommutative Orlicz spaces}


\author{Turdebek N. Bekjan}

\address{Astana IT University,  Nur-Sultan 010000, Kazakhstan}

\begin{abstract}
Let $\mathcal{M}$  be a $\sigma$-finite von Neumann algebra equipped with a
 normal  faithful state $\varphi$,  and let $\Phi$ be a growth function. We consider  Haagerup noncommutative Orlicz spaces  $L^\Phi(\M,\varphi)$ associated with $\M$ and $\varphi$, which are analogues of Haagerup
$L^p$-spaces. We show that $L^\Phi(\M,\varphi)$  is independent of $\varphi$ up
to isometric isomorphism. We prove the Haagerup's reduction theorem and the duality theorem  for this  spaces. As application of these results, we extend some noncommutative martingale inequalities in the
tracial case to  the Haagerup noncommutative Orlicz space case.

\end{abstract}


\end{frontmatter}

\section{Introduction}

The theory of noncommutative $L^p$-spaces (in the tracial case) was begin in the 20th century 50's by Segal \cite{S} and Dixmier \cite{Di} (for more details see \cite{FK,PX}). Its extension is the theory of Orlicz spaces assoiated with a semi-finite von Neumann algebra $\M$, which was presented  in \cite{Bi,Ku,M1,M2} (see also \cite{AB}).  Al-Rashed and Zegarli$\rm\acute{n}$ski \cite{AZ} introduced  the noncommutative Orlicz spaces associated
to a normal faithful state on a semifinite von Neumann algebra. In \cite{ACA},  the authors considered a certain class of non commutative Orlicz spaces,
associated with arbitrary faithful normal locally-finite weights on a semi-finite von Neumann algebra $\M$.

Around 1980's the theory of noncommutative $L^p$-spaces generalized
to type III von Neumann algebras due to the efforts of Haagerup \cite{H1},
Hilsum \cite{Hi}, Araki and Masuda \cite{AM}, Kosaki \cite{Ko} and Terp \cite{Te1} (also see \cite{H2,PX,Te}). Labuschagne \cite{L} constructed analogues of noncommutative Orlicz spaces  for type III algebras. This
spaces have some properties  of Haagerup
$L^p$-spaces, but it is not clear that this Orlicz spaces coincide with the usual in the tracial
case and  duality theorem holds for this spaces.

The main purpose of this article is to establish the basic theory of Haagerup's non-commutative Orlicz space. Applying some techniques in \cite{ACA}, some ideas in \cite{L} and the concept of weakly noncommutative Orlicz spaces in \cite{BCLJ}, we define a family of  noncommutative Orlicz spaces $L^\Phi(\M,\varphi)$ that is associated to every $\sigma$-finite von Neumann algebra $\mathcal{M}$  equipped with a
 normal  faithful state $\varphi$  for any choice of growth function  $\Phi$, which is an analogue of the family of Haagerup
$L^p$-spaces. The key aspect
of this construction that $L^\Phi(\M,\varphi)$  is independent of $\varphi$ up
to isometric isomorphism. We extend the Haagerup's reduction theorem  to the Orlicz spaces case. Using this result and convergence of noncommutative martingales, we proved the duality theorem for $L^\Phi(\M,\varphi)$. As application of this reduction theorem, we extend some  theorems  for
noncommutative martingales to   Haagerup noncommutative Orlicz spaces.

The organization of the paper is as follows. In Section 2, we
give some definitions and  related results of  noncommutative Orlicz spaces associated with growth functions, noncommutative  Orlicz spaces associated with a weight and noncommutative  weak Orlicz spaces associated with a weight. We define  Haagerup noncommutative Orlicz spaces  in Section 3. In Section 4, we give reduction theorem for Haagerup noncommutative Orlicz spaces. The duality theorem for $L^\Phi(\M,\varphi)$ will prove in Section 5. Finally, in Section 6, we prove some  theorems  of
noncommutative martingales in  Haagerup noncommutative Orlicz space case.

\section{Preliminaries}

If $\Phi$ is a
continuous and nondecreasing function from $[0,\infty)$ onto itself, then $\Phi$ is called  a growth function.

Throughout this paper, we always assume that for a growth function $\Phi$, $t\Phi'(t)$ is also a growth function. A growth function  $\Phi$ is said to satisfy the $\bigtriangleup_{2}$-condition  for all $t$, written  as $\Phi\in\bigtriangleup_{2}$, if there is $K>2$ such
 that $\Phi(2t)\leq K\Phi(t)$ for  all  $t\ge0 $. It is easy to check that $\Phi \in \triangle_2$ if and only if for any $a > 0$ there is a constant $C_a>0$ such that $\Phi(a t)\leq C_a \Phi(t)$ for all $t>0$.

  $\Phi$ is called  to satisfy the $\Delta_\frac{1}{2}$-condition
 for all $t$, written  as $\Phi\in\Delta_\frac{1}{2}$, if there is a constant $0<K<1$ such that $\Phi(\frac{t}{2})\leq K\Phi(t)$ for all $t\ge0$.

For a growth function $\Phi$, set $a=\sup\{t:\;\Phi(t)=0\}$. Then $a<\infty$ and $\Phi(t)=0$ for all $t\in [0,a]$. Without of loss generality, we may assume that $\Phi(t)>0$ for all $t>0$ (otherwise replace $\Phi$ by $\Phi(a+\cdot)$). We define
\be
a_{\Phi} = \inf_{t>0} \frac{t \Phi'(t)}{\Phi (t)}\quad \text{and} \quad b_{\Phi} = \sup_{t>0} \frac{t \Phi'(t)}{\Phi (t)}.
\ee
Then  $\Phi \in \triangle_2$  if and only if $b_{\Phi}< \8$, and $\Phi\in\Delta_\frac{1}{2}$ if and only if $a_{\Phi}>0$.

 Let  $\Phi$ be a growth function,  $1\leq n<\infty$. We set $\Phi^{(n)}(t)=\Phi(t^n)$. Then
$a_{\Phi^{(n)}}=np_\Phi$, $b_{\Phi^{(n)}}=nq_\Phi$.
If $\Phi\in\Delta_2\cap\Delta_\frac{1}{2}$,
then there exists a natural number $n\in \mathbb{N}$ such that $\Phi^{(m)}$ is
a convex growth function for every $m\ge n$ (see \cite[Theorem 1.5]{AB}).

For a growth function $\Phi\in\Delta_2\cap\Delta_\frac{1}{2}$,  the  inverse $\Phi^{-1}$ is uniquely defined on $[0,\8)$.  Set
\be
\varphi_\Phi(t)=\left\{\begin{array}{cc}
                         \frac{1}{\Phi^{-1}(\frac{1}{t})} & \mbox{if}\;t>0 \\
                         0 &\mbox{if}\; t=0
                       \end{array}
\right..
\ee

In what follows,  we always denote by  $\Phi$ a growth function satisfies $\Phi\in\Delta_2\cap\Delta_\frac{1}{2}$.


\subsection{ Noncommutative Orlicz spaces}

In this paper $\R$ always denotes  a semifinite  von Neumann algebra on
the Hilbert space $\mathcal{H}$ with a faithful normal semifinite  trace
$\nu$. Let
 $L_0(\R)$ be the set of all $\nu$-measurable operators. Then  $L_0(\R)$ is a $\ast$-algebra with respect to the strong sum and strong product
(see  \cite{FK,PX}.  Set $R^+=\{x\in\R:\;x\ge0\}$ and $L_0(\R)^+=\{x\in L_0(\R):\;x\ge0\}$.

Let $\P(\R)$ be the lattice of projections of $\R$. For given $\varepsilon>0$ and $\delta>0$, we define
\be
V(\varepsilon, \delta)=\left\{\begin{array}{c}
x\in L_0(\R)\;:\; \exists\;e\in\P(\R)\;
  \mbox{such that}\;
   e(H)\subset D(x),\\ \|xe\|\le\varepsilon\ \mbox{and}\
  \tau(e^\perp)\le\delta  \end{array}\right\}.
\ee

We denote by $P_\mathcal{K}$ the
orthogonal projection from $\mathcal{H}$ onto a closed subspace $\mathcal{K}\subset
\mathcal{H}$.  Let $x\in \M$ and $x=u|x|$ be its the  polar
decomposition. Then $u^*u=P_{(\ker x)^\perp}$  and $uu^*=P_{\overline{{\rm im} x}}$
(with ${\rm im} x=x(\mathcal{H})$). We call $r(x)=u^*u$ and $\ell(x)=uu^*$  are
the right and  left supports of
$x$, respectively. It is clear that if x is self-adjoint, then $r(x) = l(x)$.  We call this common projection is  the support of
$x$ and denote by $s(x)$.

Let $S^+(\R)=\{x\in\R^+\;:\; \tau(s(x))<\8\}$ and $S(\R)$ be
the linear span of $S^+(\R)$.

For $x\in L_0(\R)$, we define the   distribution function $\lambda (x)$ of $x$ as following: $\lambda_{t}(x)=\nu(e_{(t,\infty)}(|x|))$ for $t>0,$ where $e_{(t,\infty)}(|x|)$ is the spectral projection of $|x|$ in the interval $(t,\infty),$ and the  the   generalized singular
numbers   $\mu (x)$ of $x$  as $\mu_t(x)=\inf\{s>0:\; \lambda_s(x)\le t\}$ for $t >0$. The $\mu(x)$ is continuous from the right on $(0,\8)$.

\beq\label{eq:generalized singular}
\mu_t(x)=\inf\big\{\|xe\|\;:\; e\in\P,\;\tau(e^\perp)\le
 t\big\}= \inf\big\{\varepsilon\;:\; x\in V(\varepsilon,t)\big\},
\eeq
where for any projection $e$ we let $e^\perp=1-e$
 (for more details see \cite{FK}).
\begin{definition}\label{Orlicz spaces}
The noncommutative Orlicz space on $(\R,\nu)$ defined by
\be
L^\Phi(\R)=\{x\in
L_0(\R):\nu(\Phi(|x|))=\int_0^{\nu(1)}\Phi(\mu_t(x))dt<\infty\}.
\ee
 For
$x\in
L^\Phi(\R)$,
we define
$$
\|x\|_{\Phi}=\inf\{\lambda>0:\nu(\Phi(\frac{|x|}{\lambda}))\le1\}.
$$
\end{definition}
Using the same method in the proof of \cite[Lemma 2.3 and Theorem 2.4]{AB}, we obtain that
 $L^\Phi(\R)$ is a quasi-Banach  space, and if $x\in L^{\Phi^{(p)}}(\R),$
$y\in L^{\Phi^{(q)}}(\R)$, then
\beq\label{eq:holder}
\|xy\|_{\Phi^{(r)}}\le\|x\|_{\Phi^{(p)}}\|y\|_{\Phi^{(q)}},
\eeq
where $0< r,p,q\leq\infty$ and $r^{-1}=p^{-1}+q^{-1}$.

The next result is similar to \cite[Proposition 2]{ACA}. For easy reference, we give its proof.
\begin{lemma}\label{lem:density-nu-finite}
The $S(\R)$ is dense in $L^{\Phi}(\R)$.
\end{lemma}
\begin{proof} $\R\cap L^{\Phi}(\R)$ is dense in $L^{\Phi}(\R)$. Indeed,
set $e_n=e_{(0,n]}(|x|)$ and $x_n=xe_n$, for any $n\in\mathbb{N}$.  Then $x_n\in\R$ and
$x-x_n=xe^\bot_n$. Therefore, $(x-x_n)e_n=0$. Since $x$
is measurable, $\lim_{n\rightarrow\8}\nu(e^\bot_n)=0$. It follows that  $xe_{n}- x\rightarrow0$ in measure as $n\rightarrow\8$.
Using \cite[Lemma 3.1]{FK}, we get for any $t>0$, $\mu_t(xe_{n}- x)\rightarrow0$  as $n\rightarrow\8$. On the other hand, by \cite[Lemma 2.5 (vi)]{FK},  $\mu_t(xe_{n}- x)\le\mu_t(x)$ for all $t>0$. Applying \cite[Corollary 2.8]{FK} and Lebesgue dominated convergence theorem, we get
\be
\lim_{n\rightarrow\8}\tau(\Phi(|xe_n-x|))=\lim_{n\rightarrow\8}\int_{0}^{\8}\Phi(\mu_t(xe_n-x)dt=0.
\ee
Therefore, $\lim_{n\rightarrow\8}\|xe_n-x\|_\Phi=0$. Thus $\R\cap L^{\Phi}(\R)$ is dense in $L^{\Phi}(\R)$.

Let  $x\in S(\R)$. By \cite[Lemma 2.5(i), 2.6 and Corollary 2.8]{FK}, we have that
\be
\nu(\Phi(|x|)=\int_{0}^{\nu(s(x))}\Phi(\mu_t(x))dt\le\int_{0}^{\nu(s(x))}\Phi(\|x\|)dt<\8.
\ee
Hence, $S(\R)\subset \R\cap L^{\Phi}(\R)$. So, it sufficient to prove
that $S(\R)$ is dense in $\R\cap L^{\Phi}(\R)$.  Let $x\in \R\cap L^{\Phi}(\R)$, and
let $f_n=e_{(0,\frac{1}{n}]}(|x|)$, for any $n\in\mathbb{N}$. Then $|x|f_n$ decreases to $0$ in the norm of $\R$ as $n\rightarrow\8$, and so
$\mu_t(|x|f_n)$ decreases to $0$ as $n\rightarrow\8$ for any $t>0$. From this it follows that $\nu(\Phi(|x|f_n)\rightarrow0$  as $n\rightarrow\8$. Hence, $\lim_{n\rightarrow\8}\|x-xf_n^\bot\|_\Phi=\lim_{n\rightarrow\8}\|xf_n\|_\Phi=0$. On the other hand,
 $xf_n^\bot\in S(\R)$. Thus, $S(\R)$ is dense in $\R\cap L^{\Phi}(\R)$.
\end{proof}

\subsection{Noncommutative Orlicz spaces associated with a weight}

 We denote by $L_{loc}(\R)$ by the set of all
locally-measurable  operators affiliated with $\R$. It is well-known that $L_{loc}(\R)$ is a $\ast$-algebra with respect to the strong sum and strong product and
$L_{0}(\R)$  is a $*$-subalgebra in $L_{loc}(\R)$ (see \cite{MC1,MC2}). Set  $L_{loc}(\R)^+=\{x\in L_{loc}(\R):\;x\ge0\}$. Let
\be
\tilde{\nu}(x)=\sup\{\nu(y):\;y\in \R^+,\;y\le x\}, \qquad x\in L_{loc}(\R)^+.
\ee
Then $\tilde{\nu}$ is an extension of $\nu$ to $L_{loc}(\R)^+$ (see \cite[$\S$4.1]{MC2}). The extension  will be denoted still by $\nu$.

\begin{definition}
 \begin{enumerate}[\rm(i)]
 \item A {\it   weight} on $\R$ is a map $\omega: \R^+\to [0, \infty]$
satisfying
 \be
 \omega(x+\lambda y)=\omega(x)+\lambda\omega(y),\qquad\forall\;x, y\in \R^+,\ \forall\; \lambda\in\mathbb{R}
\ee
(where $0.\8=0$).
\item A  weight $\omega$ is said to be {\it normal} if
$\sup_i\omega(x_i)=\omega(\sup_i x_i)$ for any bounded increasing net
$(x_i)$ in $\R^+$, {\it faithful} if $\omega(x)=0$ implies $x=0$, {\it
semifinite} if the linear span $\R_\omega$ of the cone $\R_\omega^+=\{x \in\R^+:\; \omega(x)<\8\}$ is
dense in $\R$ with respect to the ultra-weak topology, and {\it locally finite} if for any non-zero
$x\in\R^+$ there is a non-zero $y\in\R^+$ such that $y\le x$ and
$0<\omega(y)<\infty$.
\end{enumerate}
\end{definition}
Let $\omega$  be a faithful normal semifinite weight on $\R$. Then  $\omega$  has a Radon-Nikodym derivative $D_{\omega}$ with respect to
$\nu$ such that $\omega(\cdot)=\nu(D_{\omega}\cdot)$ (see \cite{PT}).
The weight $\omega$ is locally finite if and only if the operator  $D_{\omega}$  is locally measurable (see \cite{Tr}). We recall that normal locally finite weight is semi-finite (see \cite[Lemma 1]{Tr}).  In the sequel, unless otherwise specified, we always denote by  $\omega$  a faithful normal locally finite weight on $\R$.
 Let $\Phi^{-1}:[0,\8)\rightarrow[0,\8)$ be the inverse of $\Phi$.  If $\alpha\in[0,1]$, then for any $x\in L_{loc}(\R)$,
\be
\Phi^{-1}(D_{\omega})^\alpha x\Phi^{-1}(D_{\omega})^{1-\alpha},\;\Phi\big(|\Phi^{-1}(D_{\omega})^\alpha x\Phi^{-1}(D_{\omega})^{1-\alpha}|\big)\in L_{loc}(\R)
\ee
(see \cite[Section 5]{MC1}).
Set
\be
\R^{\Phi}_{\alpha,\omega}=\left\{
\begin{array}{l}
 \Phi^{-1}(D_{\omega})^\alpha x\Phi^{-1}(D_{\omega})^{1-\alpha}:\;x\in\R, \\
 \Phi^{-1}(D_{\omega})^\alpha x\Phi^{-1}(D_{\omega})^{1-\alpha}
\in L^\Phi(\R)
\end{array}
\right\}
\ee
and
\be
\|\Phi^{-1}(D_{\omega})^\alpha x\Phi^{-1}(D_{\omega})^{1-\alpha}\|_{\Phi,\alpha,\omega}=\|\Phi^{-1}(D_{\omega})^\alpha x\Phi^{-1}(D_{\omega})^{1-\alpha}\|_\Phi.
\ee
Then $(\R^{\Phi}_{\alpha,\omega}, \|\cdot\|_{\Phi,\alpha,\omega})$ is a quasi-normed  space.

\begin{definition} Let $\omega$  be a faithful normal semifinite weight on $\R$ and $\alpha\in[0,1]$. The completion of $(\R^{\Phi}_{\alpha,\omega}, \|\cdot\|_{\Phi,\alpha,\omega})$ is called the Orlicz space associated with $\Phi,\;\R$ and $\omega$, it is denoted by $L^{\Phi}_{\alpha,\omega}(\R)$.
\end{definition}

Similar to the N-function case in  \cite{ACA}, we have the following results.
\begin{lemma}\label{lem:density-orlicz} If $D$ be a positive nonsingular operator  in $L^1(\R)$ and $\alpha\in[0,1]$, then $\Phi^{-1}(D)^{\alpha}\R\Phi^{-1}(D)^{1-\alpha}$ is dense in $L^{\Phi}(\R)$.
\end{lemma}
\begin{proof} Set $e_n=e_{(\frac{1}{n},n]}(D)$, for any $n\in\mathbb{N}$. Then $e_n$ increases strongly to $1$ and $\tau(e_n)<\8$, for any $n\in\mathbb{N}$. Let $x\in S(\R)$. Then there is a projection $e$ in $\R$ such that $\tau(e)<\8$ and $ex=xe=x$. Hence, $x\in L^1(\R)$. By \cite[Lemma 2.1]{J}, we get
 $\lim_{n\rightarrow\8}\|xe_{n}-x\|_{1}=0$. It follows that  $xe_{n}- x\rightarrow0$ in measure as $n\rightarrow\8$.
By \cite[Lemma 3.1]{FK}, we get for any $t>0$, $\mu_t(xe_{n}- x)\rightarrow0$  as $n\rightarrow\8$. Using \cite[Lemma 2.5 (vi)]{FK},  we obtain that  $\mu_t(xe_{n}- x)\le\mu_t(x)$ for all $t>0$. Hence, by \cite[Corollary 2.8]{FK} and Lebesgue dominated convergence theorem, we get
\be
\lim_{n\rightarrow\8}\nu(\Phi(|xe_n-x|))=\lim_{n\rightarrow\8}\int_{0}^{\nu(e)}\Phi(\mu_t(xe_n-x)dt=0.
\ee
Therefore, $\lim_{n\rightarrow\8}\|xe_n-x\|_{\Phi}=0$. Similarly,
$\lim_{n\rightarrow\8}\|e_n x-x\|_{\Phi}=0$. So,
\be
\lim_{n\rightarrow\8}\|e_nxe_n-x\|_{\Phi}\leq K[\lim_{n\rightarrow\8}\|xe_n-x\|_{\Phi}+\lim_{n\rightarrow\8}\|e_{n}x-x\|_{\Phi}]=0,
\ee
where $K$ is the constant of the quasi-norm $\|\cdot\|_\Phi$.
It follows that the closure of $\cup_{n=1}^{\8}e_n\R e_n$ in $L^{\Phi}(\R)$ contains $S(\R)$. By Lemma \ref{lem:density-nu-finite}, $\cup_{n=1}^{\8}e_n\R e_n$ is dense in $L^{\Phi}(\R)$.

Next, we prove that $\Phi^{-1}(D)^{\alpha}\R\Phi^{-1}(D)^{1-\alpha}\subset L^{\Phi}(\R)$.  Let $y\in\R$. If $\alpha\in(0,1)$, then
\be
\lim_{n\rightarrow\8}\nu(\Phi^{(\frac{1}{\alpha})}(|\Phi^{-1}(D)^{\alpha}-e_n\Phi^{-1}(D)^{\alpha}|)=\lim_{n\rightarrow\8}\nu(D-De_n)=0.
\ee
It follows that $\lim_{n\rightarrow\8}\|\Phi^{-1}(D)^{\alpha}-e_n\Phi^{-1}(D)^{\alpha}\|_{\Phi^{(\frac{1}{\alpha})}}=0$. By \eqref{eq:holder}, we get
\be
\begin{array}{rl}
&\|\Phi^{-1}(D)^{\alpha}y\Phi^{-1}(D)^{1-\alpha}-e_n\Phi^{-1}(D)^{\alpha}y\Phi^{-1}(D)^{1-\alpha}\|_{\Phi}\\
&\qquad\le\|(1-e_n)\Phi^{-1}(D)^{\alpha}\|_{\Phi^{(\frac{1}{\alpha})}}\|y\Phi^{-1}(D)^{1-\alpha}\|_{\Phi^{(\frac{1}{1-\alpha})}}\\
&\qquad\le \|y\|\|(1-e_n)\Phi^{-1}(D)^{\alpha}\|_{\Phi^{(\frac{1}{\alpha})}}\|\Phi^{-1}(D)^{1-\alpha}\|_{\Phi^{(\frac{1}{1-\alpha})}}.
\end{array}
\ee
Hence, $\lim_{n\rightarrow\8}\|\Phi^{-1}(D)^{\alpha}y\Phi^{-1}(D)^{1-\alpha}-e_n\Phi^{-1}(D)^{\alpha}y\Phi^{-1}(D)^{1-\alpha}\|_{\Phi}=0$. Similarly, $\lim_{n\rightarrow\8}\|\Phi^{-1}(D)^{\alpha}y\Phi^{-1}(D)^{1-\alpha}-\Phi^{-1}(D)^{\alpha}y\Phi^{-1}(D)^{1-\alpha}e_n\|_{\Phi}=0$. From these it follows that $\lim_{n\rightarrow\8}\|\Phi^{-1}(D)^{\alpha}y\Phi^{-1}(D)^{1-\alpha}-e_n\Phi^{-1}(D)^{\alpha}y\Phi^{-1}(D)^{1-\alpha}e_n\|_{\Phi}=0$. On the other hand,
$e_n\Phi^{-1}(D)^{\alpha},\;\Phi^{-1}(D)^{1-\alpha}e_n\in\R$, and so for all $n\in\mathbb{N}$,
\be
e_n\Phi^{-1}(D)^{\alpha}y\Phi^{-1}(D)^{1-\alpha}e_n\in e_n\R e_n.
\ee
 Therefore, $\Phi^{-1}(D)^{\alpha}y\Phi^{-1}(D)^{1-\alpha}\in L^{\Phi}(\R)$. In the case $\alpha=0,1$, this result also holds. Thus $\Phi^{-1}(D)^{\alpha}\R\Phi^{-1}(D)^{1-\alpha}\subset L^{\Phi}(\R)$.

Finally, we prove that $\Phi^{-1}(D)^{\alpha}\R\Phi^{-1}(D)^{1-\alpha}$ is dense in $L^{\Phi}(\R)$. For this it is sufficient to prove that $\cup_{n=1}^{\8}e_n\R e_n\subset\Phi^{-1}(D)^{\alpha}\R\Phi^{-1}(D)^{1-\alpha}$. Since for any $ n\in\mathbb{N}$, $e_n\Phi^{-1}(D)^{-\alpha},\;\Phi^{-1}(D)^{\alpha-1}e_n\in\R$, we have that
\be
\begin{array}{rl}
e_n\R e_n&=\Phi^{-1}(D)^{\alpha}\Phi^{-1}(D)^{-\alpha}e_n\R e_n\Phi^{-1}(D)^{\alpha-1}\Phi^{-1}(D)^{1-\alpha}\\
&=\Phi^{-1}(D)^{\alpha}e_n\Phi^{-1}(D)^{-\alpha}\R \Phi^{-1}(D)^{\alpha-1}e_n\Phi^{-1}(D)^{1-\alpha}\\
&\subset\Phi^{-1}(D)^{\alpha}\R\Phi^{-1}(D)^{1-\alpha}.
\end{array}
\ee
It follows that $\cup_{n=1}^{\8}e_n\R e_n\subset\Phi^{-1}(D)^{\alpha}\R\Phi^{-1}(D)^{1-\alpha}$.
\end{proof}

\begin{theorem}\label{thm:eqivalent-orliczspace} If $\omega$  is a faithful normal semifinite weight on $\R$ such that its the Radon-Nikodym derivative $D_{\omega}$ with respect to
$\nu$ satisfy $D_{\omega}\in L^1(\R)$ and $\alpha\in[0,1]$, then $L^{\Phi}_{\alpha,\omega}(\R)$ and  $L^{\Phi}(\R)$  are isometrically isomorphic.
\end{theorem}
\begin{proof} We define $T_\alpha:\R^{\Phi}_{\alpha,\omega}\rightarrow \Phi^{-1}(D_{\omega})^{\alpha}\R\Phi^{-1}(D_{\omega})^{1-\alpha}$ by
\be
T_\alpha(\Phi^{-1}(D_{\omega})^{\alpha}x\Phi^{-1}(D_{\omega})^{1-\alpha})=\Phi^{-1}(D_{\omega})^{\alpha}x\Phi^{-1}(D_{\omega})^{1-\alpha},
\ee
for any $\Phi^{-1}(D_{\omega})^{\alpha}x\Phi^{-1}(D_{\omega})^{1-\alpha}\in\R^{\Phi}_{\alpha,\omega}$.
Then $T_\alpha$ is a linear isometry from $\R^{\Phi}_{\alpha,\omega}$ to $\Phi^{-1}(D_{\omega})^{\alpha}\R\Phi^{-1}(D_{\omega})^{1-\alpha}$. By the definition of $L^{\Phi}_{\alpha,\omega}(\R)$ and Lemma \ref{lem:density-orlicz}, we know that $\R^{\Phi}_{\alpha,\omega}$ is dense in $L^{\Phi}_{\alpha,\omega}(\R)$ and  $\Phi^{-1}(D_{\omega})^{\alpha}\R\Phi^{-1}(D_{\omega})^{1-\alpha}$ is dense in  $L^{\Phi}(\R)$. Hence, we can extend $T_\alpha$ to an isometric isomorphism between $L^{\Phi}_{\alpha,\omega}(\R)$ and  $L^{\Phi}(\R)$.
\end{proof}

\subsection{Noncommutative  weak Orlicz spaces associated with a weight}

In  \cite{BM1}, the noncommutative  weak Orlicz spaces associated with a weight has been defined for N-function and gave its some properties. Next, we extend  it for the
growth function case.
The noncommutative weak Orlicz space $L^{\Phi,\8} (\R)$ is defined as following:
\be
L^{\Phi,\8} (\R) = \left\{ x \in L_{0}(\R):\;  \sup_{t > 0}
t \Phi( \mu_t (x)) < \8 \right\},
\ee
equipped with
\be
\|x\|_{\Phi,\8} = \inf \left \{c>0:\; t \Phi( \mu_t (x)/c ) \leq 1,\; \forall t>0 \right\}.
\ee
It is clear that for $x\in L^{\Phi,\8} (\R)$,
\beq\label{eq:DistrSingularEquiv}
\|x\|_{\Phi,\8} = \inf \left\{c>0:\; \varphi_\Phi(t) \mu_t (x)/c  \leq 1,\; \forall t>0 \right\}=\sup_{t>0}\varphi_\Phi(t) \mu_t (x).
\eeq
Using the same method as in the prof of  \cite[Proposition 3.3]{BCLJ}, we obtain that  $L^{\Phi,\8} (\R)$  is a quasi-Banach space.
We denote the closure
 of $S(\R)$ in $L^{\Phi,\8}(\R)$ by $L^{\Phi,\8}_0(\R)$.

\begin{lemma}\label{lem:holder-weak}   Let $0< r,p,q\leq\infty$ such
that $r^{-1}=p^{-1}+q^{-1}$. Then  there is a constant $C>0$ such that for any
$x\in L^{\Phi^{(p)},\8}(\mathcal{M})$ and $y\in L^{\Phi^{(q)},\8}(\mathcal{M})$,
\be
\|xy\|_{\Phi^{(r)},\8}\le C\|x\|_{\Phi^{(p)},\8}\|y\|_{\Phi^{(q)},\8}.
\ee
\end{lemma}
\begin{proof} Let $x\in L^{\Phi^{(p)},\8}(\mathcal{M})$ and $y\in L^{\Phi^{(q)},\8}(\mathcal{M})$. There exists a natural number $n\in \mathbb{N}$ such that $\Phi^{(n)}$ is
a convex growth function.
By \eqref{eq:DistrSingularEquiv} and \cite[Lemma 2.5 (vii)]{FK}, we get
\be
\begin{array}{rl}
 \sup_{t>0}t \Phi^{(r)}( \mu_t (xy))&= \sup_{t>0}2t \Phi^{(r)}( \mu_{2t} (xy))\\
 &\le2\sup_{t>0}t \Phi( \mu_{t} (x)^r\mu_{t} (y)^r) \\
&=2\sup_{t>0}t \Phi^{(n)}(\mu_{t}(x)^\frac{r}{n}\mu_{t}(y)^\frac{r}{n})\\
&\leq2\sup_{t>0}t \Phi^{(n)}(\frac{r}{p}\mu_t(x)^{\frac{p}{n}}+\frac{r}{q}\mu_t(y)^{\frac{q}{n}})\\
&\le 2[\frac{r}{p}\sup_{t>0}t \Phi^{(p)}(\mu_t(x))+\frac{r}{q}\sup_{t>0}t \Phi^{(q)}(\mu_t(y))]\\
&\le \frac{r}{p}\sup_{t>0}t \Phi^{(p)}(\mu_t(2^{\frac{n}{p}}x))+\frac{r}{q}\sup_{t>0}t \Phi^{(q)}(\mu_t(2^{\frac{n}{q}}y))
\end{array}
\ee
Therefore,
\be
\sup_{t>0}t \Phi^{(r)}( \mu_t (\frac{xy}{2^\frac{n}{r}\|x\|_{\Phi^{(p)},\8}\|y\|_{(\Phi)^{(q)},\8}}))\le1.
\ee
So, it follows that $\|xy\|_{\Phi^{(r)},\8}\le 2^\frac{n}{r}\|x\|_{\Phi^{(p)},\8}\|y\|_{(\Phi)^{(q)},\8}$.
\end{proof}

Let
\be
\R^{\Phi,\8}_{\alpha,\omega}=\left\{
\begin{array}{l}
\Phi^{-1}(D_{\omega})^{\alpha}x\Phi^{-1}(D_{\omega})^{1-\alpha}:\;x\in\R,\\
\Phi^{-1}(D_{\omega})^{\alpha}x\Phi^{-1}(D_{\omega})^{1-\alpha}\in L^{\Phi,\8} (\R)
\end{array}
\right\}
\ee
and
\be
\|\Phi^{-1}(D)^{\alpha}x\Phi^{-1}(D)^{1-\alpha}\|_{\Phi,\8,\alpha,\omega}=\|\Phi^{-1}(D_{\omega})^\alpha x\Phi^{-1}(D_{\omega})^{1-\alpha}\|_{\Phi,\8},
\ee
where $D_{\omega}$ as in Subsection 2.2.
Then $\R^{\Phi,\8}_{\alpha,\omega}$ is a quasi normed space.
Similar to \cite[Proposition 3.2]{BCLJ}, we have that for any $x \in \R^{\Phi}_{\alpha,\omega}$,
\beq\label{eq:olicz-weakorlicz}
\| x\|_{\Phi,\8,\alpha,\omega} \le \| x \|_{\Phi,\alpha,\omega}.
\eeq

\begin{definition} Let $\omega$  be a faithful normal semifinite weight on $\R$ and $\alpha\in[0,1]$. We call the completion of $(\R^{\Phi,\8}_{\alpha,\omega}, \|\cdot\|_{\Phi,\8,\alpha,\omega})$ is  the weak  Orlicz space associated with $\Phi,\R$ and $\omega$, denote by $L^{\Phi,\8}_{\alpha,\omega}(\R)$.
\end{definition}

We will need the following results (for more details see \cite{BM1}).

\begin{lemma}\label{lem:density}
Let $D$ be a positive nonsingular operator  in $L^1(\R)$. If $\alpha\in[0,1]$, then $\Phi^{-1}(D)^{\alpha}\R\Phi^{-1}(D)^{1-\alpha}$ is dense in $L^{\Phi,\8}_0(\R)$.
\end{lemma}

\begin{theorem}\label{thm:eqivalent-weakspace} Let $\omega$  be a faithful normal semifinite weight on $\R$ such that its the Radon-Nikodym derivative $D_{\omega}$ with respect to
satisfy $D_{\omega}\in L^1(\R)$.
 If $\alpha\in[0,1]$, then $L^{\Phi,\8}_{\alpha,\omega}(\R,\nu)$ and  $L^{\Phi,\8}_0(\R)$  are isometrically isomorphic.
\end{theorem}

\subsection{Noncommutative martingale spaces}

In what follows, we always denote $(\R_{n})_{n\geq 1}$ an increasing sequence of von Neumann subalgebras of $\R$ whose union  $\cup_{n\geq 1} \R_{n}$ generates $\R$ in the $w^*$-topology. For every $n \ge 1,$ the restriction $\nu |_{\R_n}$ of $\nu$ to $\R_{n}$ remains semi-finite, still denoted by $\nu,$ and we assume that there exists a trace preserving conditional expectation $\E_n$ from $\R$ onto $\R_n.$ In this case, $(\R_{n})_{n\geq 1}$ is called a filtration of $\R$. Note that $\E_n$ extends to a contractive projection from $L^p(\R)$ onto $L^p(\R_n)$ for all $1\le p\le\8$.

We denote by $L_0[0, \gamma)$ the space of Lebesgue measurable ($\mu$-measurable) real-valued functions $f$ on $[0, \gamma)$ such that $\mu(\{\omega \in [0,\gamma) :\; |f(\omega)| > s\})<\8$ for some $s$, where $0<\gamma\le\8$. The decreasing rearrangement function $f^*: [0, \gamma) \mapsto [0, \gamma)$ for $f \in L_0 [0, \gamma)$ is defined by
$$
f^*(t) = \inf\{s > 0 : \; \mu ( \{\omega \in [0,\gamma) :\; |f (\omega)| > s\}) \le t\}
$$
for $t \ge 0$. If $f, g \in L_0[0, \gamma)$ such that $\int_{0}^{t} f^* (s) ds \le \int_{0}^{t}g^* (s) ds$ for all $t \ge 0$, $f$ is said to be {\it majorized} by $g$, denoted by $f \preccurlyeq g$.

 Let $E$ be a Banach subspace of $L_0 [0, \gamma),$ simply called a Banach function space on $[0,\gamma)$ in the sequel. $E$ is said to be  symmetric if, for $f\in E$ and $g\in L_0[0,\gamma)$ such that $g^* (t) \le f^*(t)$ for all $t \ge 0,$ one has $g\in E$ and $\|g\|_{E}\le\|f\|_{E}$. A  symmetric   Banach function space $E$ on $[0,\gamma)$  is said to have the Fatou property if for every net $(f_{i})_{i\in I}$ in $E$ satisfying
$0\le f_i\uparrow$   and $\sup_{i\in I}\|f_i\|_E < \infty$, the supremum $x =\sup_{i\in I} f_i$ exists in $E$ and $\|f_i\|_E\uparrow\|f\|_E$. We say that $E$ has order continuous norm if for every net $(f_{i})_{i\in I}$ in $E$ such that $f_i\downarrow0$ we have $\|f_i\|_E\downarrow0$.
The K\"{o}the dual of $E$ is given by
$$
E^{\times} =\bigg \{f \in  L_0[0,\gamma):\; \sup_{\|g\|_E \le 1} \int_0^\gamma |f(t) g(t)| dt < \infty \bigg \}
$$
with the norm $\|f\|_{E^\times} = \sup_{\|g\|_E \le 1} \int_0^\gamma |f(t)g(t)| dt$. $E^{\times}$ is a symmetric Banach function space on $[0,\gamma),$ has the Fatou property.
A symmetric Banach space $E$ on $[0,\gamma)$ has the Fatou property
if and only if $ E = E^{\times\times}$ isometrically. It has order continuous norm if and only if
it is separable, which is also equivalent to the statement $E^{*} = E^{\times}$

Let $E$ be a symmetric   Banach function space on $[0,\gamma)$. Recall that  $E$ is called an interpolation space in $(L^p[0,\gamma),L^q[0,\gamma))\;(1\le p\le q\le8)$, if
\be
L^p[0,\gamma) \cap L^q[0,\gamma)\subset E \subset L^p[0,\gamma) + L^q[0,\gamma),
\ee
and  $T$ is a bounded linear operator from $L^p[0,\gamma) + L^q[0,\gamma)$ to $L^p[0,\gamma) + L^q[0,\gamma)$ such that
\be
T [L^p[0,\gamma)] \subset L^p[0,\gamma)\qquad\mbox{and}\qquad T [L^q[0,\gamma)] \subset L^q[0,\gamma),
\ee
we have $T [E] \subset E$ and
\be
\| T\|_{E\rightarrow E} \le C (\|T\|_{L^p[0,\gamma)\rightarrow L^p[0,\gamma)} + \|T\|_{L^q[0,\gamma)\rightarrow L^q[0,\gamma)}),
\ee
 for some $C>0$. In this case, we write $E \in \mathrm{Int}(L^p,L^q)$ (cf. \cite{BS,BL,KPS}).

 We define
\be
E(\R)=\{x  \in L_0(\R) :\;\mu_{.}(x)\in
E\}\quad(\nu(1)=\gamma);
\ee
\be
\|x\|_{E(\R)}=\|\mu_{.}(x)\|_{E}, \quad \quad x
\in E(\R).
\ee
Then $(E(\R),\|.\|_{E(\R)})$ is a  Banach
space (  \cite{DDP1,X}).

If $E$ is a symmetric Banach function space on $[0,\gamma)$ that belongs to $\mathrm{Int}(L^1,L^\8)$, then for every $n \ge 1$, $\E_n$ is bounded from $E(\R)$ onto $E(\R_n)$.

A noncommutative martingale with respect to $(\R_{n})_{n\geq 1}$ is a sequence $x=(x_{n})_{n\geq 1}$ in $L_1 (\R) + \R$ such that $\E_n(x_{n+1})= x_n$ for any $n \ge 1$. If in addition, all $x_n$'s are in $E(\R)$ then $x$ is called an $E(\R)$-martingale. In this case, if $\|x\|_E=\sup_{n\geq 1}\|x_{n}\|_{E(\R)} <\infty$, then $x$ is said to be a bounded $E(\R)$-martingale.

 For two nonnegative  quantities $A$ and $B,$ by $A \lesssim B$ we mean that there exists an absolute constant $C>0$ such that $A \leq C B,$ and by $A \approx B$ that $A \lesssim B$ and $B \lesssim A$.

  For $x_\8\in E(\R)$, we denote $x_n = \E_n(x_\8)\;(n \ge 1)$. Then $x=(x_{n})_{n\geq 1}$ is  a bounded $E(\R)$-martingale. We
observe that conversely, if $E \in \mathrm{Int}(L^p, L^q)$ for $1 < p \le q <\8$ and satisfy the Fatou property, then any bounded $E(\R)$-martingale $x = (x_n)_{n\ge1}$ is of the form $(\E_n(x_\8))_{n\ge1}$ where $x_\8\in E(\R)$ satisfying
$\|x\|_E\thickapprox \|x_\8\|_{E(\R)}$, with equality if $E$ is an exact interpolation space (see \cite{RWX}).

\section{Haagerup noncommutative Orlicz spaces}

We keep all notations introduced in the last section. Our references  for modular theory  are \cite{PT,T}, for the Haagerup noncommutative
$L^p$-spaces are \cite{H1,Te}.
 In this paper, we always assume that $\mathcal{M}$  is  a $\sigma$-finite von Neumann algebra on a complex
Hilbert space $\mathcal{H}$, equipped with a distinguished
 normal  faithful state $\varphi$.
 The one parameter modular automorphism
group of $\mathcal{M}$ associated with $\varphi$ denoted by  $\{\sigma_{t}^\varphi\}_{t\in\real}$. Let
$$
\mathcal{N}=\mathcal{M}\rtimes_{\sigma^\varphi}\real
$$
be the crossed product of $\mathcal{M}$ by
$\{\sigma_{t}^\varphi\}_{t\in\mathbb{R}}$. Then
$\mathcal{N}$ is the semi-finite von Newmann algebra acting on the Hilbert space
$L^{2}(\mathbb{R},\mathcal{H})$, generated by
$$
\left\{\pi(x):\;x\in\mathcal{M}\right\}\cup\left\{\lambda(s):\;s\in \mathbb{R}\right\},
$$
where $\pi(x)$ and $\lambda(s)$ defined as following: for all $\xi\in
L^{2}(\mathbb{R},\mathcal{H})$ and for all $t\in\mathbb{R}$,
$$
(\pi(x)\xi)(t)=\sigma_{-t}^\varphi(x)\xi(t),\qquad (\lambda(s)\xi)(t)=\xi(t-s).
$$
The operators $\pi(x)$ and $\lambda(t)$ satisfy
\be
\lambda(t)\pi(x)\lambda(t)^*=\pi(\sigma_{t}^\varphi(x)), \qquad\forall t\in \real, \;\forall x\in\M.
\ee
Hence,
$$
\sigma_{t}^\varphi(x)=\lambda(t)x\lambda^{\ast}(t), \qquad x\in\mathcal{M},\; t\in\mathbb{R}.
$$
We identify $\mathcal{M}$ with its image $\pi(\mathcal{M})$ in $\mathcal{N}$.

Let $\{\hat{\sigma}_{t}\}_{t\in\mathbb{R}}$ be the dual action of $\mathbb{R}$
on $\mathcal{N}$. Then it is a one parameter automorphism group of $\mathbb{R}$ on
$\mathcal{N},$ implemented by the unitary representation
$\{W_{t}\}_{t\in\mathbb{R}}$ of $\mathbb{R}$ on $L^{2}(\mathbb{R},\mathcal{H}:$
\beq\label{dual action}
\hat{\sigma}_{t}(x)=W(t)xW^{\ast}(t),\qquad \forall x\in\mathcal{N},\;\forall t\in\mathbb{R},
\eeq
where the  $W(t)$ is defined by
$$
W(t)(\xi)(s)=e^{-its}\xi(s),\qquad \forall\xi\in L^{2}(\mathbb{R},\mathcal{H}),\;\forall
s,\;t\in\mathbb{R}.
$$
Recall that the dual action $\hat{\sigma}_{t}$ is uniquely determined by the
following conditions: for any $x\in\M$ and $s\in\mathbb{R}$,
$$
\hat{\sigma}_{t}(x)=x \qquad\mbox{and}\qquad
\hat{\sigma}_{t}(\lambda(s))=e^{-ist}\lambda(s),\quad \forall t\in\mathbb{R}.
$$
Therefore,
$$
\mathcal{M}=\{x\in\mathcal{N}: \; \hat{\sigma}_{t}(x)=x,\;\forall
t\in\mathbb{R}\}.
$$
On the other hand, there exists the unique
 normal semi-finite faithful trace $\tau$ on $\mathcal{N}$ satisfying
$$
\tau\circ\hat{\sigma}_{t}=e^{-t}\tau,\qquad \forall t\in\mathbb{R}.
$$
Since any  normal semi-finite faithful weight $\psi$ on
$\mathcal{M}$ induces a dual normal semi-finite weight $\hat{\psi}$ on
$\mathcal{N}$. Then $\hat{\psi}$ admits a Radon-Nikodym derivative with
respect to $\tau$ (cf. \cite{PT}). Hence, the dual weight
$\hat{\varphi}$ of our distinguished state $\varphi$ has the
Radon-Nikodym derivative $D$ with respect to $\tau$. Then
\beq\label{eq:D-tau}
\hat{\varphi}(x)=\tau(Dx),\qquad x\in \mathcal{N}^{+}.
\eeq
Recall that $D$ is an invertible positive selfadjoint operator on
$L^{2}(\mathbb{R},\mathcal{H}),$ affiliated with $\mathcal{N}$, and that the
regular representation $\lambda(t)$ above is given by
\beq\label{eq:lambda}
\lambda(t)=D^{it},\qquad \forall t\in \mathbb{R}.
\eeq
Recall that $L_{0}(\mathcal{N},\tau)$
denotes the topological $\ast$-algebra of all operators on
$L^{2}(\mathbb{R},\mathcal{H})$ measurable with respect to
$(\mathcal{N},\tau).$ Then the Haagerup noncommutative $L^{p}$-spaces,
$0<p\leq\infty$, are defined by
$$
L^{p}(\mathcal{M},\varphi)=\{x\in L_{0}(\mathcal{N},\tau):
\;\hat{\sigma}_{t}(x)=e^{-\frac{t}{p}}x,\;\forall t\in\mathbb{R}\}.
$$
The $L^{p}(\mathcal{M},\varphi)$ are closed selfadjoint linear subspaces of
$L_{0}(\mathcal{N},\tau)$.  Let $x\in L^{p}(\mathcal{M},\varphi)$ and $x=u|x|$
be its polar decomposition. Then
 $$
 u\in\mathcal{M}\qquad\mbox{and}\qquad|x|\in L^{p}(\mathcal{M},\varphi).
 $$
Recall that
$$
L^{\infty}(\mathcal{M},\varphi)=\mathcal{M}.
$$
As mentioned previously, for any $\psi\in\mathcal{M}_{\ast}^{+}$, its dual
weight $\hat{\psi}$ has a Radon-Nikodym derivative $D_{\psi}$ with respect to
$\tau$ such that
$$
\hat{\psi}(x)=\tau(h_{\psi}x),\qquad x\in\mathcal{N}^{+}.
$$
Then $h_{\psi}\in L_{0}(\mathcal{N},\tau)$ and $\hat{\sigma}_{t}(h_{\psi})=e^{-t}h_{\psi}$ for all $t\in\mathbb{R}$.
Thus $h_{\psi}\in L^{1}(\mathcal{M},\varphi)^{+}$.
Hence, this correspondence
between $\mathcal{M}_{\ast}^{+}$ and $ L^{1}(\mathcal{M},\varphi)^{+}$ extends
to a bijection between $\mathcal{M}_{\ast}$ and $ L^{1}(\mathcal{M},\varphi).$
So that for $\psi\in\mathcal{M}_{\ast}, $ if  $\psi=u|\psi|$ is its polar
decomposition, the corresponding  $h_{\psi} \in L^{1}(\mathcal{N},\tau)$ admits
the polar decomposition
\beq\label{eq:polardecomposition-haagerupL1}
h_{\psi}=u|h_{\psi}|=uh_{|\psi|}.
\eeq
Then the norm on
$L^{1}(\mathcal{M},\varphi),$ is defined as
\beq\label{eq:haagerup-L1}
\parallel h_{\psi}\parallel_{1}=|\psi|(1)=\parallel
\psi\parallel_{\ast},\qquad \forall\psi\in\mathcal{M}_{\ast}.
\eeq
In this way,
$$
L^{1}(\mathcal{M},\varphi)=\mathcal{M}_{\ast}\;\mbox{isometrically}.
$$
For
$0<p<\infty$, set
$$
\parallel x\parallel_{p}=\| |x|^{p}\|_{1}^{\frac{1}{p}},\qquad \forall x\in
L^{p}(\mathcal{M},\varphi).
$$
Since $x\in L^{p}(\mathcal{M},\varphi)$ if and only if $|x|\in
L^{p}(\mathcal{M},\varphi)$, it follows that for $1\leq p<\infty$ (resp. $0<p<1$),
$$
(L^{p}(\mathcal{M},\varphi),\;\|\cdot\|_p)
$$
is a Banach space (resp. a quasi-Banach space). To describe duality of Haagerup noncommutative
$L^{p}$-space, we use the distinguished linear function on $L^{1}(\M,\varphi)$
which is defined by
\beq\label{eq:tr}
\tr(x)=\psi_{x}(1),\qquad \forall x\in L^{1}(\mathcal{M},\varphi),
\eeq
where $\psi_{x}\in\mathcal{M}_{\ast}$ is the unique normal functional
associated with $x$ by the above identification between $\mathcal{M}_{\ast}$
and $L^{1}(\M,\varphi)$. Then $tr$ is a continuous functional on
$L^{1}(\M,\varphi)$ satisfying
\beq\label{eq:trace-norm}
|\tr(x)|\leq \tr(|x|)=\| x\|_{1},\qquad \forall x\in L^{1}(\M,\varphi).
\eeq
Recall that
\beq\label{eq:trace-positive-norm}
\| x\|_{1}=\tr(x)=\tau(e_{(1,\8)}(x)),\qquad \forall x\in L^{1}(\M,\varphi)^+.
\eeq
Let
$$
1\leq p<\infty,\qquad\frac{1}{p}+\frac{1}{q}=1.
$$
Then $tr$ have the following  property:
\beq\label{eq:tr-Lp-Lq}
\tr(xy)=\tr(yx),\qquad \forall x\in L^{p}(\M,\varphi),\qquad\forall y\in
L^{q}(\M,\varphi).
\eeq
The bilinear form $(x,y)\mapsto
\tr(xy)$ defines a duality bracket between $L^{p}(\M,\varphi)$ and
$L^{q}(\M,\varphi)$, for which
\be
(L^{p}(\M,\varphi))'=L^{q}(\M,\varphi)\qquad  \mbox{isometrically for all}\;1\leq
p<\infty.
\ee

Moreover, our distinguished state $\varphi$ can be recovered from $tr$
(recalling that $D$ is the Radon-Nikodym derivative of $\hat{\varphi}$ with
respect to $\tau$),  that is,
\beq\label{eq:D-tr}
\varphi(x)=\tr(Dx),\qquad \forall x\in \mathcal{M}.
\eeq
For $0<p<\infty,\; 0\le \eta\le1$, we have that
\beq\label{eq:haagerup-dense-Lp}
L^{p}(\M,\varphi)=\mbox{the closure of}\; D^{\frac{1-\eta}{p}}\M D^{\frac{\eta}{p}}\;\mbox{in}\; L^{p}(\M,\varphi).
\eeq
For $0 < p <\8$ and any $x$ in $L^p(\M,\varphi)$, we have
\beq\label{eq:generalized singular-haagerup}
\mu_t(x)=t^{-\frac{1}{p}}\|x\|_p,\qquad t>0,
\eeq
where $\mu_t(\cdot)$  is relative to $(\N,\tau)$ (see \cite[Lemma 4.8]{FK}).

From \eqref{eq:generalized singular-haagerup} and \cite[Theorem 3.1]{L} it follows the following result.

\begin{theorem}\label{thm:Lp-convergence-measuere}
Let $0 < p \le\8$. Then $\mu_1(x)=\|x\|_p$ for all $x\in L^{p}(\M,\varphi)$.
In addition the uniform topology on $L^{p}(\M,\varphi)$ is homeomorphic to the topology of convergence in measure that $L^{p}(\M)$ inherits from $L_0(\N)$.
\end{theorem}

Since $D\in L^1(\N)$, we have that $\Phi^{-1}(D)^\alpha\M\Phi^{-1}(D)^{1-\alpha} \subset L^{\Phi,\8}(\N)$.
Indeed, by [23, Lemma 2.5 (iv), (vi)and (vii)], for any $x\in\M$,
$$
\begin{array}{rl}
    \mu_t (\Phi^{-1}(D)^\alpha x\Phi^{-1}(D)^{1-\alpha} ) & \leq \mu_{\frac{t}{2}} (\Phi^{-1}(D)^\alpha  )\mu_{\frac{t}{2}} ( x\Phi^{-1}(D)^{1-\alpha} ) \\
    & \leq \mu_{\frac{t}{2}} (\Phi^{-1}(D)^\alpha  )\|x\|\mu_{\frac{t}{2}} ( \Phi^{-1}(D)^{1-\alpha} )\\
    & \leq \|x\|\Phi^{-1}(\mu_{\frac{t}{2}} (D))^\alpha \Phi^{-1}( \mu_{\frac{t}{2}} (D) )^{1-\alpha}\\
    &=\|x\|\Phi^{-1}(\mu_{\frac{t}{2}} (D))
  \end{array}
$$
Using  $\Phi\in\bigtriangleup_{2}$, we obtain that
$$
\begin{array}{rl}
\sup_{t > 0}t \Phi(\mu_t (\Phi^{-1}(D)^\alpha x\Phi^{-1}(D)^{1-\alpha} ))&\le \sup_{t > 0}t \Phi(\|x\|\Phi^{-1}(\mu_{\frac{t}{2}} (D)))\\
&\le C_{\|x\|}\sup_{t > 0}t \mu_{\frac{t}{2}} (D))\\
&\le 2C_{\|x\|}\|D\|_1<\8.
\end{array}
$$
Hence, $\Phi^{-1}(D)^\alpha x\Phi^{-1}(D)^{1-\alpha} \in L^{\Phi,\8}(\N)$.

For $x\in\M$, we define
\be
\|\Phi^{-1}(D)^{\alpha}x\Phi^{-1}(D)^{1-\alpha}\|_{\Phi,\alpha} =\|\Phi^{-1}(D)^{\alpha}x\Phi^{-1}(D)^{1-\alpha}\|_{\Phi,\8}.
\ee

\begin{definition}  The completion of $(\Phi^{-1}(D)^{\alpha}\M\Phi^{-1}(D)^{1-\alpha}, \|\cdot\|_{\Phi,\alpha})$  is called the Haagerup noncommutative Orlicz space associted with $\Phi,\M$ and $\varphi$, we denote it  by $L^{\Phi,\alpha}(\M,\varphi)$.
\end{definition}

If $\Phi(t)=t^p$ with $0< p < \8$,  then from \eqref{eq:DistrSingularEquiv} and \eqref{eq:generalized singular-haagerup} it follows that
\be
L^{\Phi,\alpha}(\M,\varphi)=L^{p}(\M,\varphi)
\ee
By Lemma \ref{lem:density}, we get
\beq\label{eq:haagerub-subspace-weakorlicz}
L^{\Phi,\alpha}(\M,\varphi)\subset L^{\Phi,\8}_0(\N)
\eeq

\begin{proposition}\label{pro:haagerup-weak} Let $x\in L^{\Phi,\alpha}(\M,\varphi)$.
If  $\Phi(|x|)\in L^1(\M,\varphi)$, then
$$
\|x\|_{\Phi,\alpha} =\inf\{\lambda>0:\|\Phi\big(|\frac{x}{\lambda}|\big)\|_1\le1\}.
$$
\end{proposition}
\begin{proof}  Since $\Phi\big(|x|\big)\in L^1(\M,\varphi)$, by \eqref{eq:generalized singular-haagerup}, we have that for any $t>0$,
\be
\mu_t(\Phi\big(|x|\big))=\frac{1}{t}\|\Phi\big(|x|\big)\|_1.
\ee
Let $\lambda>0$ such that $\|\Phi\big(|\frac{x}{\lambda}|\big)\|_1\le1$. Then
\be
\mu_t(\Phi\big(|\frac{x}{\lambda}|\big))\le \frac{1}{t}.
\ee
Using \cite[Lemma 2.5 (iv)]{FK}, we get that
\be
\sup_{t>0}\frac{1}{\Phi^{-1}(\frac{1}{t})}\mu_t(|x|)\le\lambda.
\ee
By \eqref{eq:DistrSingularEquiv}, $\|x\|_{\Phi,\8}\le\lambda$, and so
\be
\|x\|_{\Phi,\alpha}\le\inf\{\lambda>0:\|\Phi\big(|\frac{x}{\lambda}|\big)\|_1\le1\}.
\ee

Conversely, let $\lambda=\|x\|_{\Phi,\8}$. Applying \eqref{eq:DistrSingularEquiv} and \cite[Lemma 2.5 (iv)]{FK}, we deduce that
\be
t\mu_t(\Phi\big(|\frac{x}{\lambda}|\big))\le1,\qquad \forall t>0.
\ee
By \eqref{eq:generalized singular-haagerup},
$\|\Phi\big(|\frac{x}{\lambda}|\big)\|_1\le1$.
Hence,
\be
\inf\{\lambda>0:\|\Phi\big(| \frac{x}{\lambda}|\big)\|_1\le1\}\le\|x\|_{\Phi,\alpha}.
\ee
\end{proof}

\begin{remark}\label{rk:equivalent} Since $(\Phi^{-1}(D)^{\alpha}\M\Phi^{-1}(D)^{1-\alpha})^*=\Phi^{-1}(D)^{1-\alpha}\M\Phi^{-1}(D)^{\alpha}$, it follows that $L^{\Phi,\alpha}(\M,\varphi)=L^{\Phi,1-\alpha}(\M,\varphi)$.
\end{remark}

Next, in the finite von Neumann
algebras case, we will show in detail the relation
between  Haagerup noncommutative Orlicz spaces and the usual noncommutative Orlicz spaces. In the sequel, $\M$ always denotes a finite von Neumann algebra with a normal faithful trace $\nu$. Let $\phi$ be a normal faithful state
 on $\M$.  Then  there is a unique operator $g_{\phi}\in L_0(\M,\nu)^+$ such that $\phi(\cdot)=\nu(\cdot g_{\phi})$ (see \cite{PT}).
 Recall the automorphisms $\{\sigma_{t}^\nu\}_{t\in\real}$, $\{\sigma_{t}^\phi\}_{t\in\real}$
on $\M$, their dual
actions $\{\hat{\sigma}_{t}^\nu\}_{t\in\real}$  on $\mathcal{M}\rtimes_{\sigma^\nu}\real$ and $\{\hat{\sigma}_{t}^\phi\}_{t\in\real}$  on $\mathcal{M}\rtimes_{\sigma^\phi}\real$, and the
 normal semi-finite faithful traces $\tau_\nu$ on $\mathcal{M}\rtimes_{\sigma^\nu}\real$ and  $\tau_\phi$ on $\mathcal{M}\rtimes_{\sigma^\phi}\real$.
Let $\hat{\nu}$ be the dual weight of $\nu$  on
$\mathcal{M}\rtimes_{\sigma^\nu}\real$ and $\hat{\phi}$ be the dual weight of $\phi$  on
$\mathcal{M}\rtimes_{\sigma^\phi}\real$, and let  $D_{\nu}$ be the Radon-Nikodym derivative of $\hat{\nu}$ with
respect to $\tau_\nu$ and  $D_{\phi}$ be the Radon-Nikodym derivative of $\hat{\phi}$ with
respect to $\tau_\phi$.

\begin{theorem}\label{thm:haagerup-finite-vn} Let $0\le\alpha\le1$.  Then
\be
L^{\Phi,\alpha}(\M,\nu)=L^{\Phi}(\M)\otimes\Phi^{-1}(e^\cdot)=\bigg \{x\otimes\Phi^{-1}(e^\cdot):\;x\in L^{\Phi}(\M)\bigg \}
\ee
and
\be
\|x\otimes \Phi^{-1}(e^\cdot)\|_{\Phi,\alpha}=\|x\|_\Phi,\qquad \forall x\in L^{\Phi}(\M).
\ee
\end{theorem}
\begin{proof}
It is known that $\mathcal{M}\rtimes_{\sigma^\nu}\real$ will up to Fourier transform correspond to $\M\otimes L^\8(\real)$, $\tau_\nu=\nu\otimes\int_{\real}\cdot e^{-t}dt$ and $D_{\nu}=1_{id}\otimes e^\cdot$. For any  $x\in\M$, we have that
\be
\begin{array}{rl}
\Phi^{-1}(D_\nu)^{\alpha}(x\otimes1)\Phi^{-1}(D_\nu)^{1-\alpha}&=(1_{id}\otimes\Phi^{-1}(e^\cdot)^{\alpha})(x\otimes1)1_{id}(\otimes\Phi^{-1}(e^\cdot)^{1-\alpha})\\
&=x\otimes\Phi^{-1}(e^\cdot).
\end{array}
\ee
It follows that
\be
\Phi\big(|\Phi^{-1}(D_\nu)^\alpha (x\otimes1)\Phi^{-1}(D_\nu)^{1-\alpha}|\big)&=\Phi(|x|\otimes\Phi^{-1}(e^\cdot))=\Phi(|x|)\otimes e^\cdot.
\ee
Since $\Phi(|x|)\in L^1(\M)$,
\be
\Phi(|x|)\otimes e^\cdot\in L^1(\M,\nu)\qquad\mbox{and}\qquad \|\Phi(|x|)\otimes e^\cdot\|_1=\|\Phi(|x|)\|_1
\ee
(see \cite[p. 62-63]{Te}).
Applying Proposition \ref{pro:haagerup-weak}, we get
\be
\begin{array}{rl}
\|x\otimes\Phi^{-1}(e^\cdot)\|_{\Phi,\alpha}&=\inf\{\lambda>0:\|\Phi\big(\frac{|x\otimes\Phi^{-1}(e^\cdot)|}{\lambda}\big)\|_1\le1\}\\
 &=\inf\{\lambda>0:\|\Phi\big(|\frac{x}{\lambda}\otimes\Phi^{-1}(e^\cdot)|\big)\|_1\le1\}\\
&=\inf\{\lambda>0:\|\Phi\big(\frac{|x|}{\lambda}\big)\|_1\le1\}\\
&=\inf\{\lambda>0:\nu(\Phi(\frac{|x|}{\lambda}))\le1\}=\|x\|_\Phi.
\end{array}
\ee
Since $\Phi^{-1}(D_\nu)^{\alpha}(\M\otimes1)\Phi^{-1}(D_\nu)^{1-\alpha}$ and $\M$ are dense in $L^{\Phi,\alpha}(\M,\nu)$ and $L^{\Phi}(\M)$  respectively, from the above equality we obtain the desired result.
\end{proof}

Denote the Coones derivative $[D_{\phi}:D_\nu]_t=u_t, \;\forall t\in \mathbb{R}$. Then $u_t=g_{\phi}^{it}$.  We define a unitary $u$ on $L^{2}(\mathbb{R},\mathcal{H})$ by
 \be
 (u\xi)(t)=g_{\phi}^{-it}\xi(t),\qquad t\in \mathbb{R},\qquad\xi\in L^{2}(\mathbb{R},\mathcal{H}).
 \ee
Let
\be
\theta:B(L^{2}(\mathbb{R},\mathcal{H}))\rightarrow B(L^{2}(\mathbb{R},\mathcal{H})),\qquad x\mapsto uxu^*,\qquad x\in B(L^{2}(\mathbb{R},\mathcal{H})).
\ee
Then
\beq\label{eq:theta}
\theta(x\otimes 1)=\pi_{\sigma^\phi}(x),\qquad \theta(1\otimes e^{it\cdot})=\pi_{\sigma^\phi}(g_{\phi}^{-it})D_{\phi}^{it},\qquad x\in\M,\qquad t\in \mathbb{R},
\eeq
where
\be
(\pi_{\sigma^\phi}(x)\xi)(t)=\sigma_{-t}^{\phi}(x)\xi(t),\qquad \xi\in
L^{2}(\mathbb{R},\mathcal{H}),\qquad t\in\mathbb{R}.
\ee
Recall that  $\theta$ defines a topological *-isomorphism  from
$\mathcal{M}\rtimes_{\sigma^\nu}\real$ to $\mathcal{M}\rtimes_{\sigma^{\phi}}\real$  such that $\tau_\nu=\tau_\phi\circ\theta$ (see \cite[Theorem 37]{Te}).

Let
\be
V_\nu(\varepsilon, \delta)=\left\{\begin{array}{c}
x\in L_0(\mathcal{M}\rtimes_{\sigma^\nu}\real)\;:\; \exists\;e\in\P(\mathcal{M}\rtimes_{\sigma^\nu}\real)\;
  \mbox{such that}\\
   e(H)\subset D(x),\: \|xe\|\le\varepsilon\ \mbox{and}\
  \tau_\nu(e^\perp)\le\delta  \end{array}\right\}
\ee
and
\be
V_{\phi}(\varepsilon, \delta)=\left\{\begin{array}{c}
x\in L_0(\mathcal{M}\rtimes_{\sigma^{\phi}}\real)\;:\; \exists\;e\in\P(\mathcal{M}\rtimes_{\sigma^{\phi}}\real)\;
  \mbox{such that}\\
   e(H)\subset D(x),\: \|xe\|\le\varepsilon\ \mbox{and}\
  \tau_\phi(e^\perp)\le\delta  \end{array}\right\}.
\ee
By \cite[Corollary II.38]{Te}, we know that the mapping
$\theta:\;\mathcal{M}\rtimes_{\sigma^\nu}\real\rightarrow\mathcal{M}\rtimes_{\sigma^{\phi}}\real$
extends to a topological $*$-isomorphism from
$L_0(\mathcal{M}\rtimes_{\sigma^\nu}\real)$ onto $L_0(\mathcal{M}\rtimes_{\sigma^{\phi}}\real)$. We still denote this extension by $\theta$, then it satisfies   that
\be
\theta(V_\nu(\varepsilon, \delta))=V_{\phi}(\varepsilon, \delta),\qquad \forall \varepsilon>0,\;\delta>0.
\ee
Hence, by \eqref{eq:generalized singular},
\beq\label{eq:singular-value}
\mu_t(x)=\mu_t(\theta(x)),\qquad \forall x\in L_0(\mathcal{M}\rtimes_{\sigma^\nu}\real),\; t>0.
\eeq

Since
\be
\theta(g_{\phi}\otimes e^\cdot)^{is}=\theta(g_{\phi}^{is}\otimes e^{is\cdot})=\theta((g_{\phi}^{is}\otimes1)(1\otimes e^{is\cdot})),
\ee
 by \eqref{eq:lambda} and \eqref{eq:theta}, we get
\be
\theta(g_{\phi}\otimes e^\cdot)^{is}=\theta(g_{\phi}^{is}\otimes1)\theta(1\otimes e^{is\cdot})=\pi(g_{\phi}^{is})\pi(g_{\phi}^{-is})D_{\phi}^{is}=D_{\phi}^{is},\qquad s\in \mathbb{R}.
\ee
Thus
\beq\label{eq:tilde-theta}
\theta(g_{\phi}\otimes e^\cdot)=D_{\phi}.
\eeq
\begin{theorem}\label{thm:haagerup-finite-state} Let $0\le\alpha\le1$.  Then for any $x\in\M$,
\be
\Phi^{-1}(D_\phi)^{\alpha}\pi_{\sigma^\phi}(x)\Phi^{-1}(D_\phi)^{1-\alpha}=\theta((\Phi^{-1}(g_{\phi})^{\alpha}x\Phi^{-1}(g_{\phi})^{1-\alpha})\otimes\Phi^{-1}(e^\cdot))
\ee
and
\be
\|\Phi^{-1}(D_\phi)^{\alpha}\pi_{\sigma^\phi}(x)\Phi^{-1}(D_\phi)^{1-\alpha}\|_{\Phi,\alpha}=\|\Phi^{-1}(g_{\phi})^{\alpha}x\Phi^{-1}(g_{\phi})^{1-\alpha}\|_\Phi.
\ee
Consequently, $\theta$ defines an $*$-isometric isomorphism from $L^{\Phi,\alpha}(\M,\nu)$ to $L^{\Phi,\alpha}(\M,\phi)$.
\end{theorem}
\begin{proof}
Using \eqref{eq:tilde-theta}, we get
\be
\begin{array}{rl}
   &\theta\big(\Phi^{-1}(g_{\phi})^{\alpha}x\Phi^{-1}(g_{\phi})^{1-\alpha}\otimes\Phi^{-1}(e^\cdot)\big) \\ &\qquad=\theta\big((\Phi^{-1}(g_{\phi})^{\alpha}\otimes\Phi^{-1}(e^\cdot)^{\alpha})(x\otimes1)(\Phi^{-1}(g_{\phi})^{1-\alpha}\otimes\Phi^{-1}(e^\cdot)^{1-\alpha})\big)  \\
   &\qquad=\theta\big(\Phi^{-1}(g_{\phi}\otimes e^\cdot)^{\alpha}(x\otimes1)\Phi^{-1}(g_{\phi}\otimes e^\cdot)^{1-\alpha}\big)  \\
   &\qquad=\Phi^{-1}\big(\theta(g_{\phi}\otimes e^\cdot )\big)^{\alpha}\theta(x\otimes1)\Phi^{-1}\big(\theta(g_{\phi}\otimes e^\cdot)\big)^{1-\alpha}  \\
   &\qquad=\Phi^{-1}(D_{\phi})^{\alpha}\pi_{\sigma^\phi}(x)\Phi^{-1}(D_{\phi})^{1-\alpha} \\
    &\qquad=\Phi^{-1}(D_{\phi})^{\alpha}\pi_{\sigma^\phi}(x)(x)\Phi^{-1}(D_{\phi})^{1-\alpha}.
\end{array}
\ee

By \eqref{eq:singular-value} and Theorem \ref{thm:haagerup-finite-vn}, we have that
\be
\begin{array}{rl}
&\|\Phi^{-1}(D_\phi)^{\alpha}\pi_{\sigma^\phi}(x)\Phi^{-1}(D_\phi)^{1-\alpha}\|_{\Phi,\alpha}\\
&\qquad=\|(\Phi^{-1}(g_{\phi})^{\alpha}x\Phi^{-1}(g_{\phi})^{1-\alpha})\otimes\Phi^{-1}(e^\cdot)\|_{\Phi,\alpha}\\
&\qquad=\|\Phi^{-1}(g_{\phi})^{\alpha}x\Phi^{-1}(g_{\phi})^{1-\alpha}\|_\Phi.
\end{array}
\ee
\end{proof}

By Theorem \ref{thm:eqivalent-orliczspace} and \ref{thm:haagerup-finite-state}, we get the following result.

\begin{corollary}\label{cor:isometric}
Let  $0 \le \alpha \le1$. Then   $L^{\Phi,\alpha}(\M,\phi)$ is isometrically isomorphic to $L^{\Phi,1}(\M,\phi)$.
\end{corollary}

As in \eqref{eq:tr}, we define the continuous functional $\tr$ on
$L^{1}(\M,\phi)$.  By \eqref{eq:D-tr},
\beq\label{eq:tr-trace-connection}
\tr_\phi(xD_{\phi})=\phi(x)=\nu(xg_{\phi}),\qquad\forall x\in \M.
\eeq

\begin{proposition}\label{pro} Let $0\le\alpha\le1$. Then  $\Phi(|x|)\in L^1(\M,\phi)$ for any $x\in L^{\Phi,\alpha}(\M,\phi)$.
\end{proposition}

\begin{proof}
 By Theorem \ref{thm:haagerup-finite-vn},
 $$
L^{\Phi,\alpha}(\M,\nu)=L^{\Phi}(\M)\otimes\Phi^{-1}(e^\cdot)=\bigg \{x\otimes\Phi^{-1}(e^\cdot):\;x\in L^{\Phi}(\M)\bigg \}
$$
and
$$
\|x\otimes \Phi^{-1}(e^\cdot)\|_{\Phi,\alpha}=\|x\|_\Phi,\qquad \forall x\in L^{\Phi}(\M).
$$
Hence, for any $y\in L^{\Phi,\alpha}(\M,\nu)$, there exists $x\in L^{\Phi}(\M)$ such that
$y=x\otimes\Phi^{-1}(e^\cdot)$. Therefore, $\Phi(|y|)=\Phi(|x|)\otimes e^t$. On the other hand,
 $$
L^1(\M,\nu)=L^1(\M)\otimes e^\cdot=\bigg \{z\otimes e^\cdot:\;z\in L^1(\M)\bigg \}
$$
and
$$
\|z\otimes \Phi^{-1}(e^\cdot)\|_1=\|z\|_1,\qquad \forall z\in L^1(\M).
$$
Since $\Phi(|x|)\in L^1(\M)$,
\begin{equation}\label{l_1}
\Phi(|y|)\in L^1(\M,\nu).
\end{equation}

Note that  the mapping
 $\theta$ is an isometric isomorphism to from $L^{\Phi,\alpha}(\M,\nu)$ onto $L^{\Phi,\alpha}(\M,\phi)$.
By \eqref{l_1}, for any $x\in L^{\Phi,\phi}(\M,\alpha)$, we have that $\theta^{-1}(\Phi(|x|))=\Phi(\theta^{-1}(|x|))\in L^1(\M,\nu)$. Recall that $\theta$ is also isometric isomorphism  from $L^1(\M,\nu)$ onto $L^1(\M,\phi)$. Hence, we obtain that $\Phi(|x|)\in L^1(\M,\phi)$.
\end{proof}

\section{Reduction theorem for Haagerup noncommutative Orlicz spaces}

We keep all notations introduced in the preceding sections.
 Let  $\M_1$ be a von Neumann
 subalgebra of $\M$ such that $\M_1$  is invariant under $\sigma^\varphi$ i.e.,
 $$
 \sigma_{t}^\varphi(\M_1)\subset\M_1,\;\qquad \forall t\in\real.
 $$
Set $\phi=\varphi|_{\M_1}$ be the restriction of $\varphi$ to $\M_1$. Using the fact that $\M_1$ is $\sigma^\varphi$-invariant, we obtain that the modular
automorphism group associated with $\phi$ is  $\sigma^\varphi|_{\M_1}$, i.e.,
\beq\label{automorphism-equal}
\sigma^\phi_t=\sigma^\varphi_t|_{\M_1},\qquad \forall t\in\real.
\eeq
Hence the crossed product
$$
\N_1=\M_1\rtimes_{\sigma^\phi}\mathbb{R}
$$
 is a von Neumann subalgebra
of
$$
\N=\mathcal{M}\rtimes_{\sigma^\varphi}\mathbb{R}.
$$
Let $\tau_1$ be the canonical  normal semi-finite faithful trace  on $\N_1$. Then $\tau_1$ is equal to the
restriction of $\tau$ to $\N_1$ (recalling that $\tau$  is the canonical trace on $\N$).  Let $\widehat{\varphi}$ and $\widehat{\phi}$ be the dual
weights of $\varphi$ and $\phi$ respectively. By \cite{C},
 the Radon-Nikodym derivative of $\widehat{\phi}$ with respect to $\tau_1$ is equal to $D$,
the Radon-Nikodym derivative of $\widehat{\varphi}$ with respect to $\tau$. Hence, we obtain that
\be
\Phi^{-1}(D)^{\alpha}\M_1\Phi^{-1}(D)^{1-\alpha}\subset\Phi^{-1}(D)^{\alpha}\M\Phi^{-1}(D)^{1-\alpha}.
\ee
Therefore, we have the following result.
\begin{proposition}\label{pro:Haagerup-subspace}  Let $0\le\alpha\le1$. Then   $L^{\Phi,\alpha}(\M_1,\phi)$ coincides isometrically with a subspace of  $L^{\Phi,\alpha}(\M,\varphi)$.
\end{proposition}

\begin{lemma}\label{lem:convergence} If  $(a_i)_{i\in I}$ is a bounded net in $\N$ converging strongly to $a$, then $xa_i\rightarrow xa$ in the norm of
 $L^{\Phi,\8}(\N)$ for any $x\in  L^{\Phi,\8}_0(\N)$.
\end{lemma}
\begin{proof}  We choose $n\in\mathbb{N}$ such that  $\Phi^{(n)}$ is
a convex growth function and
$a_{\Phi^{(n)}}>1$. Let $E=L^{\Phi^{(n)},\8}_0[0,\gamma)$ be the weak Orlicz space on $[0,\gamma)$. Then $E$ is a symmetric Banach function space on $[0,\gamma)$ and the Boyd indices of $E$ satisfy
$1<p_E\le q_E<\8$ (see \cite[page 26-28]{M}). Let  $1< p<p_E \le q_E<q<\8$.
 By Lemma 4.5 in \cite{D}, it follows that
\beq\label{inclution}
L_p(\N)\cap L_q(\N)\subset E(\N)\subset L_p(\N)+ L_q(\N)
\eeq
with continuous inclusions.
If $y\in L_p(\N)\cap L_q(\N)$,  then using Lemma 2.3 in \cite{J}, we obtain that $ya_i\rightarrow ya$ in the norm of
$L_p(\N)$  and $ya_i\rightarrow ya$ in the norm of
$L_q(\N)$. Hence, $\lim_{i}\|ya_i-ya\|_{\Phi^{(n)},\8}=0$. Let $x\in S(\N)$, and let $x=u|x|$ be the polar decomposition of $x$. Since $|x|^\frac{1}{n}\in L_p(\N)\cap L_q(\N)$,
we have that $\lim_{i}\||x|^\frac{1}{n}a_i-|x|^\frac{1}{n}a\|_{\Phi^{(n)},\8}=0$. Therefore, by Lemma \ref{lem:holder-weak}, we deduce that
\be
\|xa_i-xa\|_{\Phi,\8}\le C\|u|x|^\frac{n-1}{n}\|_{\Phi^{(\frac{n}{n-1})},\8}\||x|^\frac{1}{n}a_i-|x|^\frac{1}{n}a\|_{\Phi^{(n)},\8}\rightarrow0.
\ee
Using the fact that $S(\N)$ is dense in $L^{\Phi,\8}_0(\N)$,  we obtain the desired result.

\end{proof}

We  need to give a brief description of  reduction theorem that
approximates a type III von Neumann algebra by finite ones  (cf. \cite{H1,H2}). Throughout this section, $G$ will denote the discrete subgroup
$\cup_{n\ge1}2^{-n}\mathbb{Z}$  of $\mathbb{R}$. We consider the discrete crossed product
\beq\label{defi:U}
\U =\mathcal{M}\rtimes_{\sigma^\varphi} G.
\eeq
 Recall that $\U$ is a von Neumann algebra on $\el_2(G,H)$
generated by the operators $\pi(x),\; x\in \M$ and $\lambda(t),\; s \in G$, which are defined by
$$
(\pi(x)\xi)(t)=\sigma_{-t}(x)\xi(t),\qquad (\lambda(s)\xi)(t)=\xi(t-s),\qquad \; \xi\in\el_2(G,H),\qquad \forall t \in G.
$$
Then $\pi$ is a normal faithful representation of $\M$ on
$\el_2(G,H)$ and we identify $\pi(\M)$ with $\M$.  The operators $\pi(x)$ and $\lambda(t)$ satisfy
the following commutation relation:
\beq\label{representation1}
\lambda(t)\pi(x)\lambda(t)^*=\pi(\sigma_{t}^\varphi(x)),\qquad\forall t\in G,\qquad\forall x\in \M.
\eeq
Let $\widehat{\varphi}$ be the dual weight of $\varphi$ on $\U$. Then $\widehat{\varphi}$ is again a faithful
normal state on $\U$ whose restriction on $\mathcal{M}$ is $\varphi$. The modular automorphism group of $\widehat{\varphi}$ is uniquely determined
by
\be
\sigma_{t}^{\widehat{\varphi}}(\pi(x))=\pi(\sigma_{t}^\varphi(x)),\qquad \sigma_{t}^{\widehat{\varphi}}(\lambda(s))=\lambda(s),\qquad x\in\M, \qquad s,\;t\in G.
\ee
Consequently, $\sigma_{t}^{\widehat{\varphi}}|_{\M}=\sigma_{t}^{\varphi}$, and so $\sigma_{t}^{\widehat{\varphi}}(\M)=\M$ for all $t\in\real$.
 We recall that there is a unique normal faithful conditional expectation $\F$ from
$\U$ onto $\M$ determined by
\beq\label{conditional-expectation1}
\F(\lambda(t)x)=\left\{\begin{array}{rl}
                           x & \mbox{if}\; t=0 \\
                           0 & \mbox{otherwise}
                         \end{array}
\right.,\qquad \forall x\in\M,\qquad\forall t\in G.
\eeq
It satisfies
\beq\label{eq:equation-dual weight-condition}
\widehat{\varphi}\circ\F = \widehat{\varphi},\qquad\sigma_{t}^{\widehat{\varphi}}\circ\F=\F\circ\sigma_{t}^{\widehat{\varphi}}, \qquad\forall t\in \real.
\eeq

We denote  by $\U_{\widehat{\varphi}}$ the centralizer of  $\widehat{\varphi}$  in $\U$. Then
$$
\U_{\widehat{\varphi}}=\{x\in\U:\qquad \sigma_{t}^{\widehat{\varphi}}(x)=x, \qquad \forall t\in \real\}.
$$
For each $n \in \mathbb{N}$ there exists a unique $b_n\in\mathcal{Z}(\U_{\widehat{\varphi}})$ such that
$$
0\le b_n\le2\pi\qquad \mbox{and}\qquad e^{ib_n}=\lambda(2^{-n}),
$$
 where $\mathcal{Z}(\U_{\widehat{\varphi}})$ the center of $\U_{\widehat{\varphi}}$. Let $a_n=2^nb_n$ and define  normal faithful positive functionals on $\U$ by
$$
\varphi_n(x)=\widehat{\varphi}(e^{-a_n}x),\qquad \forall x\in \U,\qquad \forall n\in\mathbb{N}.
$$
Then $\sigma_{t}^{\varphi_n}$ is $2^{-n}$-periodic and
\beq\label{eq:modular}
\sigma_{t}^{\varphi_n}(x)=e^{-ita_n}\sigma_{t}^{\widehat{\varphi}}(x)e^{ita_n},\qquad \forall x\in\U,\qquad\forall t\in \real,\qquad\forall n\in\mathbb{N}.
\eeq
Let $\U_n=\U_{\varphi_n}$. Then $\U_n$ is a finite von Neumann algebra equipped with the normal faithful
tracial state $\tau_n=\varphi_n|_{\U_n}$.
Define
$$
\F_n(x)=2^n\int_0^{2^{-n}}\sigma_{t}^{\varphi_n}(x)dt,\qquad \forall x\in\U.
$$
By the $2^{-n}$-periodicity of $\sigma_{t}^{\varphi_n}$, we have that
\beq\label{eq:conditionexpectation-n}
\F_n(x)=\int_0^1\sigma_{t}^{\varphi_n}(x)dt,\qquad \forall x\in\U.
\eeq
 Haagerup's reduction theorem (Theorem 2.1 and Lemma 2.7 in \cite{H2}) asserts that $\{\U_n\}_{n\ge1}$ is an
increasing sequence of von Neumann subalgebras of $\U$ with the following properties:
\begin{enumerate}[\rm(i)]
\item Each $\U_n$ is finite with the normal faithful
tracial state $\tau_n$;
\item   $\F_n$ is a faithful normal conditional expectation from $\U$ onto $\U_n$ such that
\beq\label{eq:equation-dual weight-condition-n}
\widehat{\varphi}\circ\F_n = \hat{\varphi},\quad  \sigma_{t}^{\widehat{\varphi}}\circ\F_n=\F_n\circ\sigma_{t}^{\widehat{\varphi}},\quad \F_n \circ\F_{n+1} = \F_n,\quad \forall t\in \real,\; \forall n\in\mathbb{N};
\eeq
\item For any $x\in\U$, $\F_n(x)$ converges to $x$ $\sigma$-strongly as $n\rightarrow\8$;
\item $\bigcup_{n\ge1} \U_n$ is $\sigma$-strongly dense in $\U$.
 \end{enumerate}

 Let $\mathfrak{N}=\U\rtimes_{\sigma^{\widehat{\varphi}}}\mathbb{R}$ and $\rho$ be the canonical  normal semi-finite faithful trace  on $\mathfrak{N}$. Then $\N=\mathcal{M}\rtimes_{\sigma^\varphi}\mathbb{R}$  is a von Neumann subalgebra
of $\mathfrak{N}$ and $\rho|_\N=\tau$  (recalling that $\tau$  is the canonical trace on $\N$).  Let $\widetilde{\varphi}$ and $\widetilde{\widehat{\varphi}}$ be the dual
weights of $\varphi$ and $\hat{\varphi}$ respectively. Then
 the Radon-Nikodym derivative of $\widetilde{\widehat{\varphi}}$ with respect to $\rho$ is equal to $D$ (recalling that
the Radon-Nikodym derivative of $\varphi$ with respect to $\tau$).

 \begin{theorem}\label{thm:Haagerup's reduction} Let $\U,\;\U_n,\;\F_n\; (\forall n\in \mathbb{N})$ be fixed as in the above and let $0 \le \alpha \le1$.
\begin{enumerate}[\rm(i)]
\item $\{L^{\Phi,\alpha}(\U_n,\hat{\varphi})\}_{n\ge1}$ is a  increasing sequence of subspaces of $L^{\Phi,\alpha}(\U,\hat{\varphi}))$;
\item $\bigcup_{n\ge1}L^{\Phi,\alpha}(\U_n,\hat{\varphi})$  is dense in  $L^{\Phi,\alpha}(\U,\hat{\varphi})$;
\item For each $n$, $L^{\Phi,\alpha}(\U_n,\hat{\varphi})$ is isometric to the usual
on noncommutative  Orlicz space $L^{\Phi}(\U_n)$.
\end{enumerate}
 \end{theorem}
\begin{proof}(i) By Proposition \ref{pro:Haagerup-subspace}, the result holds.

(ii) We assume that $\frac{1}{2}\le\alpha\le1$ (see Remark \ref{rk:equivalent}). If $x\in\U$, then $\F_n(x)$ converges to $x$  $\sigma$-strongly as $n\rightarrow\8$. By Lemma \ref{lem:convergence}, $\Phi^{-1}(D)^{\alpha}\F_n(x)\rightarrow \Phi^{-1}(D)^{\alpha}x$ in the norm of
 $L^{\Phi^{(\frac{1}{\alpha})},\8}(\mathfrak{N})$. Hence, using Lemma \ref{lem:holder-weak}, we get
 \be
 \Phi^{-1}(D)^{\alpha}\F_n(x)\Phi^{-1}(D)^{1-\alpha}\rightarrow \Phi^{-1}(D)^{\alpha}x\Phi^{-1}(D)^{1-\alpha}
 \ee
 in the norm of
 $L^{\Phi,\8}(\mathfrak{N})$. Therefore,  the desired result holds.

(iii) follows from Theorem \ref{thm:haagerup-finite-vn}.
\end{proof}

 Similar to Theorem \ref{thm:Lp-convergence-measuere}, we have the following result.

\begin{theorem}\label{thm:measure-convergence-equivalent} Let $0 \le \alpha \le1$.
Then
\be
\|x\|_{\Phi,\alpha}=\frac{\mu_1(x)}{\Phi^{-1}(1)},\qquad \forall x\in L^{\Phi,\alpha}(\M,\varphi).
\ee
In addition the norm topology in $L^{\Phi,\alpha}(\M)$ is homeomorphic to the topology of convergence in measure inherited from $L_0(\N)$.
\end{theorem}
\begin{proof} Let $x\in L^{\Phi,\alpha}(\M)$. By Theorem \ref{thm:Haagerup's reduction}, there is a sequence $(x_n)_{n\ge1}$ in $\bigcup_{n\ge1}L^{\Phi,\alpha}(\U_n)$ such that
\be
\|x-x_n\|_{\Phi,\alpha}=\|x-x_n\|_{\Phi,\8}\rightarrow0.
\ee
Since \cite[Proposition 3.3(1)]{BCLJ} remains valid in the case that $\Phi$ is a growth function,  it follows that $x_n \rightarrow x$ in measure. Then $\Phi(|x_n|) \rightarrow \Phi(|x|)$ in measure (see \cite{Ti}). On the other hand, from  (iii) in Theorem \ref{thm:Haagerup's reduction} and Proposition \ref{pro}, it follows that $\Phi(|x_n|) \in L^1(\U)$. Applying Theorem \ref{thm:Lp-convergence-measuere}, we obtain that $\Phi(|x|) \in L^1(\U)$.
Using Proposition \ref{pro:haagerup-weak},  \eqref{eq:generalized singular-haagerup} and \cite[Lemma 2.5]{FK}, we get that
\be
\begin{array}{rl}
\|x\|_{\Phi,\alpha} &=\inf\{\lambda>0:\|\Phi\big(|\frac{x}{\lambda}|)\|_1\le1\}\\
&=\inf\{\lambda>0:\mu_1(\Phi(|\frac{x}{\lambda}|\big))\le1\}\\
&=\inf\{\lambda>0:\Phi(\mu_1(\frac{x}{\lambda}))\le1\}\\
&=\inf\{\lambda>0:\mu_1(\frac{x}{\lambda})\le\Phi^{-1}(1)\}\\
&=\inf\{\lambda>0:\frac{\mu_1(x)}{\Phi^{-1}(1)}\le\lambda\}=\frac{\mu_1(x)}{\Phi^{-1}(1)}.
\end{array}
\ee
If $(x_n)_{n\ge1}$ is a sequence in $L^{\Phi,\alpha}(\M)$ such that $(x_n)_{n\ge1}$ converges  $x\in L^{\Phi,\alpha}(\M)$, then $\|x_n -x\|_{\Phi,\8}\rightarrow0$ as $n\rightarrow\8$. Hence, $x_n\rightarrow x$ in measure.

Conversely, if $(x_n)_{n\ge1}$ is a sequence in $L^{\Phi,\alpha}(\M)$  and  $x\in L_0(\N)$ such that $x_n \rightarrow x$ in measure, then $(x_n)_{n\ge1}$ is a Cauchy sequence in $L_0(\N)$. Using \cite[Lemma 3.1]{FK}, we know that $\mu_1(x_n-x_m)\rightarrow0$ as $n,m\rightarrow\8$.  It follows that
\be
\|x_n-x_m\|_{\Phi,\alpha}=\frac{\mu_1(x_n-x_m)}{\Phi^{-1}(1)}\rightarrow0\qquad \mbox{as}\;n,m\rightarrow\8.
\ee
Hence, there is an operator $y\in L^{\Phi,\alpha}(\M)$ such that $\|x_n -y\|_{\Phi,\8}\rightarrow0$ as $n\rightarrow\8$, and so $x_n \rightarrow y$  in measure. Therefore, $x=y$, i.e.  $x_n \rightarrow x$ in the norm of $L^{\Phi,\alpha}(\M)$.
\end{proof}

Let $\varphi'$ be another normal  faithful state on $\mathcal{M}$, and let $\{\sigma_{t}^{\varphi'}\}_{t\in\real}$ be the one parameter modular automorphism
group of $\mathcal{M}$ associated with $\varphi'$. Set
$$
\N'=\mathcal{M}\rtimes_{\sigma^{\varphi'}}\real.
$$
Let $\tau'$ be  the unique
 normal semi-finite faithful trace on  $\N'$ satisfying
$$
\tau'\circ\hat{\sigma}_{t}=e^{-t}\tau',\quad \forall t\in\mathbb{R}.
$$
We denote the Connes cocycle derivative by $[D_{\varphi'}:D_\varphi]_t=u_t\;(t\in\mathbb{R})$ and  define a unitary $u$ on $L^{2}(\mathbb{R},\mathcal{H})$ as following:
 \be
 (u\xi)(t)=u_{-t}\xi(t),\qquad t\in \mathbb{R},\quad\xi\in L^{2}(\mathbb{R},\mathcal{H}).
 \ee
Set
\be
\theta:B(L^{2}(\mathbb{R},\mathcal{H}))\rightarrow B(L^{2}(\mathbb{R},\mathcal{H})),\qquad x\mapsto uxu^*,\quad x\in B(L^{2}(\mathbb{R},\mathcal{H})).
\ee
Then $\theta$ is an automorphism and
\be
\theta(\pi(x))=\pi'(x),\quad \theta(\lambda(t))=\pi'(u^*_{t})\lambda(t);\qquad x\in\M,\quad t\in \mathbb{R},
\ee
where
\be
(\pi'(x)\xi)(t)=\sigma_{-t}^{\varphi'}(x)\xi(t),\qquad \xi\in
L^{2}(\mathbb{R},\mathcal{H}),\quad t\in\mathbb{R}.
\ee
In particular, $\theta$ defines a topological *-isomorphism  from
$\N$ to $\N'$ such that $\tau=\tau'\circ\theta$ (see \cite[Theorem II.38]{Te}).

Let
\be
V_\varphi(\varepsilon, \delta)=\left\{\begin{array}{c}
x\in L_0(\N)\;:\; \exists\;e\in\P(\mathcal{M}\rtimes_{\sigma^\varphi}\real)\;
  \mbox{such that}\\
   e(H)\subset D(x),\: \|xe\|\le\varepsilon\ \mbox{and}\
  \tau(e^\perp)\le\delta  \end{array}\right\}
\ee
and
\be
V_{\varphi'}(\varepsilon, \delta)=\left\{\begin{array}{c}
x\in L_0(\N)\;:\; \exists\;e\in\P(\mathcal{M}\rtimes_{\sigma^{\varphi'}}\real)\;
  \mbox{such that}\\
   e(H)\subset D(x),\: \|xe\|\le\varepsilon\ \mbox{and}\
  \tau(e^\perp)\le\delta  \end{array}\right\}.
\ee
By \cite[Corollary II.38]{Te}, we know that the mapping
$\theta:\;\mathcal{M}\rtimes_{\sigma^\varphi}\real\rightarrow\mathcal{M}\rtimes_{\sigma^{\varphi'}}\real$
extends to a topological *-isomorphism from
$L_0(\mathcal{M}\rtimes_{\sigma^\varphi}\real)$ onto $\rightarrow L_0(\mathcal{M}\rtimes_{\sigma^{\varphi'}}\real)$. We still denote this extension by $\theta$, then it satisfies   that
\be
\theta(V_\varphi(\varepsilon, \delta))=V_{\varphi'}(\varepsilon, \delta),\qquad \forall \varepsilon>0,\;\delta>0.
\ee
Hence, by \eqref{eq:generalized singular},
\beq\label{eq:sigularvalue-equivalent}
\mu_t(x)=\mu_t(\theta(x)),\qquad \forall x\in L_0(\mathcal{M}\rtimes_{\sigma^\varphi}\real),\; t>0.
\eeq
It follows that $\theta|_{L^{p}(\M,\varphi)}:L^{p}(\M,\varphi)\rightarrow L^{p}(\M,\varphi')\;(0<p<\8)$  is an  isometrically isomorphism.

\begin{theorem}\label{thm:independent of state} Let $0 \le \alpha \le1$.
Then
 $L^{\Phi,\alpha}(\U,\hat{\varphi})$ is isometrically isomorphic to $L^{\Phi,\alpha}(\U',\hat{\varphi}')$.
\end{theorem}
\begin{proof}
Set $\U' =\mathcal{M}\rtimes_{\sigma^{\varphi'}} G$. Denote by $\widehat{\varphi}'$  the dual weight of $\varphi'$ on $\U'$.
Now let $\mathfrak{N}'=\U'\rtimes_{\sigma^{\widehat{\varphi}'}}\mathbb{R}$ and $\rho'$ be the canonical  normal semi-finite faithful trace  on $\mathfrak{N}'$. Then $\N'$  is a von Neumann subalgebra
of $\mathfrak{N}'$ and $\rho'|_\N=\tau'$.  We denote by  $\widetilde{\varphi'}$ and $\widetilde{\widehat{\varphi}'}$  respectively the dual
weights of $\varphi'$ and $\hat{\varphi}'$. Then
 the Radon-Nikodym derivative of $\widetilde{\widehat{\varphi}'}$ with respect to $\rho'$ is equal to $D'$, where  $D'$ is
the Radon-Nikodym derivative of $\varphi'$ with respect to $\tau'$.

 Now let $\tilde{\theta}$ be an automorphism on $B(L^{2}(\mathbb{R},\mathcal{H}))$ defined by
\be
\tilde{\theta}(x)= uxu^*,\qquad x\in B(L^{2}(\mathbb{R},\mathcal{H})).
\ee
Then
\be
\tilde{\theta}(\pi_{\U}(x))=\pi'(x_{\U'}),\qquad \tilde{\theta}(\lambda(t))=\pi'(u^*_{t})\lambda(t);\qquad x\in\U,\qquad t\in \mathbb{R},
\ee
where
\be
(\pi_{\U}(x)\xi)(t)=\sigma_{-t}^{\widehat{\varphi}}(x)\xi(t),\qquad (\lambda(s)\xi)(t)=\xi(t-s),\qquad \xi\in
L^{2}(\mathbb{R},\mathcal{H}),\qquad t\in\mathbb{R}
\ee
and
\be
(\pi'_{\U'}(x)\xi)(t)=\sigma_{-t}^{{\widehat{\varphi}'}}(x)\xi(t),\qquad \xi\in
L^{2}(\mathbb{R},\mathcal{H}),\qquad t\in\mathbb{R}.
\ee
By \cite[Theorem 37]{Te}, $\theta$ defines an $*$-isomorphism  from
$\mathfrak{N}$ to $\mathfrak{N}'$ such that $\rho=\rho'\circ\theta$.  It is clear that  $\tilde{\theta}$ is an extension of $\theta$. We will denote still by $\theta$.

Let $\U_n'=\theta(\U_n)$ for any $n\in\mathbb{N}$. Define $\F'_n: \;\U'\rightarrow\U_n'$ by
\be
\F'_n(\theta(x))=\theta(\F_n(x),\qquad \forall x\in\U.
\ee
Since $\theta:\U\rightarrow\U'$ is  an $*$-isomorphism and $\sigma$-strongly continuous, we have that
\begin{enumerate}[\rm(i)]
\item Each $\U_n'$ is finite with the normal faithful
tracial state $\tau_n'$, where $\tau'_n(\theta(x))=\tau_n(x),\;\forall x\in\U$;
\item   $\F_n'$ is a faithful normal conditional expectation from $\U'$ onto $\U_n'$;
\item For any $x\in\U'$, $\F_n'(x)$ converges to $x$ $\sigma$-strongly as $n\rightarrow\8$;
\item $\bigcup_{n\ge1} \U_n'$ is $\sigma$-weakly dense in $\U'$.
 \end{enumerate}
Using the method in the proof of Theorem \ref{thm:Haagerup's reduction}, we obtain that
\begin{enumerate}[\rm(i)]
\item $\{L^{\Phi,\alpha}(\U_n',\hat{\varphi}')\}_{n\ge1}$ is a  increasing sequence of subspaces of $L^{\Phi,\alpha}(\U',\hat{\varphi}'))$;
\item $\bigcup_{n\ge1}L^{\Phi,\alpha}(\U_n',\hat{\varphi}')$  is dense in  $L^{\Phi,\alpha}(\U',\hat{\varphi}')$;
\item For each $n$, $L^{\Phi,\alpha}(\U_n',\hat{\varphi}')$ is isometric to the usual
on noncommutative  Orlicz space $L^{\Phi}(\U_n')$.
\end{enumerate}
It is clear that
 $\theta|_{L^{\Phi}(\U_n,\tau)}:L^{\Phi}(\U_n,\tau)\rightarrow L^{\Phi}(\U_n',\tau')$ is an isometrically isomorphic for each $n\in\mathbb{N}$.
Let $x\in L^{\Phi,\alpha}(\M)$. By Theorem \ref{thm:Haagerup's reduction}, there is a sequence $(x_n)_{n\ge1}$ in $\bigcup_{n\ge1}L^{\Phi,\alpha}(\U_n)$ such that
$\|x-x_n\|_{\Phi,\alpha}\rightarrow0$.
Then
\be
\|\theta(x_n)- \theta(x_m)\|_{\Phi,\alpha}=\|\theta(x_n-x_m)\|_{\Phi,\alpha}=\|x_n-x_m\|_{\Phi,\alpha}\rightarrow0\qquad\mbox{as} \;n,m\rightarrow\8.
\ee
Hence, there exists $y\in L^{\Phi,\alpha}(\U',\hat{\varphi}'))$ such that $\|\theta(x_n)- y\|_{\Phi,\alpha}\rightarrow0$. By Theorem \ref{thm:measure-convergence-equivalent}, we get that   $\theta(x_n)\rightarrow y$ in measure. On the other hand, by \eqref{eq:sigularvalue-equivalent} and \cite[Lemma 3.1]{FK}, we get $\theta(x_n)\rightarrow x$ in measure. Therefore, we deduce that
$\theta(x_n)=y$, and so $\theta(x)\in L^{\Phi,\alpha}(\U',\hat{\varphi}'))$. Using \eqref{eq:sigularvalue-equivalent} and Theorem \ref{thm:measure-convergence-equivalent}, we deduce that
$\|\theta(x)\|_{\Phi,\alpha}=\|x\|_{\Phi,\alpha}$. Since $\bigcup_{n\ge1}L^{\Phi,\alpha}(\U_n,\hat{\varphi})$ dense in $L^{\Phi,\alpha}(\U,\hat{\varphi})$ and $\bigcup_{n\ge1}L^{\Phi,\alpha}(\U_n',\hat{\varphi}')$ dense in $L^{\Phi,\alpha}(\U',\hat{\varphi}')$, we know that $\theta$ is  an isometrically isomorphic from $L^{\Phi,\alpha}(\U,\hat{\varphi})$ onto $L^{\Phi,\alpha}(\U',\hat{\varphi}')$.
\end{proof}

\begin{theorem}\label{lem:isometric}
Let $\U$ be fixed as  above and let $0 \le \alpha \le1$. Then   $L^{\Phi,\alpha}(\U)$ is isometric isomorphic to $L^{\Phi,1}(\U)$.
\end{theorem}
\begin{proof}
From the proof of Theorem \ref{thm:eqivalent-weakspace}, we know that if
\be
T_\alpha:\mathfrak{N}^{\Phi,\8}_{\alpha,\omega}\rightarrow \Phi^{-1}(D)^{\alpha}\mathfrak{N}\Phi^{-1}(D)^{1-\alpha}
\ee
 is defined  by
\be
T_\alpha(\Phi^{-1}(D)^{\alpha}x\Phi^{-1}(D)^{1-\alpha})=\Phi^{-1}(D)^{\alpha}x\Phi^{-1}(D)^{1-\alpha},
\ee
for any $\Phi^{-1}(D)^{\alpha}x\Phi^{-1}(D)^{1-\alpha}\in\mathfrak{N}^{\Phi,\8}_{\alpha,\omega}$.
Then  $T_\alpha$ extends to an isometric isomorphism between $L^{\Phi,\8}_{\alpha,\omega}(\mathfrak{N})$ and  $L^{\Phi,\8}_0(\mathfrak{N})$. The extension
will be denoted still by $T_\alpha$. Let $T=T_1^{-1}T_\alpha$. Then $T$ is an isometric isomorphic from $L^{\Phi,\8}_{\alpha,\varphi}(\mathfrak{N})$ onto  $L^{\Phi,\8}_{1,\varphi}(\mathfrak{N})$.

Let $\{L^{\Phi,\alpha}(\U_n)\}_{n\ge1}$ be the sequence in Theorem \ref{thm:Haagerup's reduction}. Then  $L^{\Phi,\alpha}(\U_n)$ is isometric to the usual
 noncommutative Orlicz space $L^{\Phi}(\U_n)$ associated with a finite trace. From the definition of $T$ it follows that $T(L^{\Phi,\alpha}(\U_n))=L^{\Phi,1}(\U_n)$. Hence,
$T|_{L^{\Phi,\alpha}(\U_n)}$ is an isometric isomorphic from $L^{\Phi,\alpha}(\U_n)$ onto $L^{\Phi,1}(\U_n)$, for all $ n\in \mathbb{N}$ (see the proof of Theorem \ref{thm:haagerup-finite-state}). Using (ii) of Theorem \ref{thm:Haagerup's reduction}, we can extend $T$ to an isometric map from  $L^{\Phi,\alpha}(\U)$ into $L^{\Phi,1}(\U)$.
Therefore,
 \beq\label{inequality2}
  \|T(x)\|_{\Phi,1}= \|x\|_{\Phi,\alpha},\quad \forall x\in  \U_n,\; n\in \mathbb{N}.
 \eeq
 Next, we only need to prove $T$ is surjective.  It is clear that $T(L^{\Phi,\alpha}(\U))$ is closed subspace of $L^{\Phi,1}(\U)$.
 Let $x\in \U$. From the proof of Theorem \ref{thm:Haagerup's reduction} (ii), we deduce that
 $$
\lim_{n\rightarrow\8}\| \F_n(x) -x\|_{\Phi,1}=0.
 $$
By \cite[Proposition 3.3]{BCLJ}, we get $\F_n(x) \rightarrow x$ in measure.
For any  $ n\in \mathbb{N}$, there is an $y_n\in L^{\Phi,\alpha}(\U_n)$ such that $T(y_n)=\F_n(x)$.

Since $\{L^{\Phi,\alpha}(\U_n)\}_{n\ge1}$ is increasing, using \eqref{inequality2}, we obtain that
$$
\|y_n-y_m\|_{\Phi,\alpha}=\|T(y_n)-T(y_m)\|_{\Phi,1}= \|\F_n(x)-\F_m(x)\|_{\Phi,1},\quad n,m\ge1.
$$
Hence, $\{y_n\}_{n\ge1}$ is a Cauchy sequence in $L^{\Phi,\alpha}(\U)$.
 So, there exists  $y$ in $L^{\Phi,\alpha}(\U)$ such that
$$
\lim_{n\rightarrow\8}\|y_n-y\|_{\Phi,\alpha}=0
$$
and  $\lim_{n\rightarrow\8}\|T(y_n)-T(y)\|_{\Phi,1}=0$. Therefore, $T(y)=x$. We  deduce that
$$
  (\U, \|\cdot\|_{\Phi,1})\subset T(L^{\Phi,\alpha}(\U)).
$$
Thus $T(L^{\Phi,\alpha}(\U))\supset L^{\Phi,1}(\U)$, and so $T$ is surjective.
\end{proof}

\section{Dual spaces of Haagerup noncommutative Orlicz spaces}

Let  $\Phi$  be an N-function, i.e., $\Phi:[0,\8)\rightarrow [0,\8) $  satisfy the following conditions:
(1) $\Phi$ is convex, (2)
$\Phi(t)=0$~iff~$t=0$,
(3)
$\lim_{t\rightarrow0}\frac{\Phi(t)}{t}=0,~\lim_{t\rightarrow
\infty}\frac{\Phi(t)}{t}=+\infty$.

Let $\Phi'(t)$ be the left derivative of $\Phi$. Then
$\Phi'(t)$  is left  continuous, nondecreasing on  $(0,\infty)$ and  satisfies:  $0<\Phi'(t)<\infty$
for  $0<t<\infty$,  $\Phi'(0)=0$ and $\lim_{t\rightarrow
\infty}\Phi'(t)=\infty$.  The left
inverse of $\Phi'$ ($\Psi'(s)=\inf\{t>0:\Phi'(t)>s\}$ for $s>0$) will be denoted by $\Psi'$. We define a complementary N-function $\Psi$ of  $\Phi$ by
\be
\Psi(s)=\int_{0}^{s}\Psi'(v)\,dv,\qquad s\ge0.
\ee
It is clear that $\Phi$ is the complementary N-function  of $\Psi$.
We call $(\Phi,\Psi)$ is a pair of complementary N-functions.

 Let $(\Phi,\Psi)$ be a pair of complementary N-functions, with inverses $\Phi^{-1},\;\Psi^{-1}$ (which are uniquely defined on $[0,\8)$). Then
 \beq\label{eq:inverses}
 t<\Phi^{-1}(t)\Psi^{-1}(t)<2t,\qquad t>0.
 \eeq

\begin{remark}\label{rk:equivalent} The definition of the N-function is a little different from the definition in \cite{RR}.
\end{remark}
 An N-function $\Phi$ is called  to satisfy the $\bigtriangledown_{2}$-condition
 for all $t$, written  as $\Phi\in\bigtriangledown_{2}$, if there is a constant $c>1$ such that $\Phi(t) \leq\frac{1}{2c}\Phi(ct)$ for all $t\ge0$.
For  a pair of complementary N-functions $(\Phi,\Psi)$, we have that $\Phi\in\bigtriangleup_{2}$ if and only if
$\Psi\in\bigtriangledown_{2}$ (see \cite[Theorem 1.1.2]{RR}).
Recall that
\beq\label{eq:relationship-index}
1 \le a_{\Phi} \le  b_{\Phi} \le \8,\qquad\frac{1}{a_{\Phi}}+\frac{1}{b_{\Psi}}=1,\qquad\frac{1}{a_{\Psi}}+\frac{1}{b_{\Phi}}=1
\eeq

\begin{lemma}\label{lem:N-function}
If $\Phi$ is a growth function, then there exists  $r>0$ such that
$\Phi^{(r)}$ is
an N-unction.
\end{lemma}
\begin{proof} Choose $r>0$ such that  $\Phi^{(r)}$ is
a convex  and $a_{\Phi^{(r)}}=ra_\Phi>1$. Let $0<t_1<t_2<\8$. Then
\be
\ln\frac{\Phi^{(r)}(t_2)}{\Phi^{(r)}(t_1)}=\int^{t_2}_{t_1}\frac{(\Phi^{(r)})'(s)}{\Phi^{(r)}(s)}ds\ge\int^{t_2}_{t_1}\frac{a_{\Phi^{(r)}}}{s}ds
=\ln\frac{t_2^{a_{\Phi^{(r)}}}}{t_1^{a_{\Phi^{(r)}}}}.
\ee
Hence,
\be
\frac{\Phi^{(r)}(t_2)}{\Phi^{(r)}(t_1)}\ge\frac{t_2^{a_{\Phi^{(r)}}}}{t_1^{a_{\Phi^{(r)}}}}.
\ee
Therefore,
\be
\lim_{t\rightarrow
\infty}\frac{\Phi^{(r)}(t)}{t}=\lim_{t_2\rightarrow
\infty}\frac{\Phi^{(r)}(t_2)}{t_2}\ge \lim_{t_2\rightarrow
\infty}\frac{\Phi^{(r)}(t_1)}{t_1^{a_{\Phi^{(r)}}}}t_2^{a_{\Phi^{(r)}}-1}=+\8
\ee
and
\be
\lim_{t\rightarrow
0}\frac{\Phi^{(r)}(t)}{t}=\lim_{t_1\rightarrow0}\frac{\Phi^{(r)}(t_1)}{t_1}\le \lim_{t_1\rightarrow
0}\frac{\Phi^{(r)}(t_2)}{t_2^{a_{\Phi^{(r)}}}}t_1^{a_{\Phi^{(r)}}-1}=0.
\ee
It follows that $\Phi^{(r)}$ is an N-function.

\end{proof}

In the rest of this paper,  we always denote by  $\Phi$ an  N-function on $(0,\8)$ such that $\Phi\in\bigtriangleup_{2}$ and by $\Psi$ its a complementary N-function.

Let $(0,\8)$ equipped with the usual Lebesgue measure $m$. We denote by $L_0(0, \8)$ the space of $m$-measurable real-valued functions $f$ on $(0, \8)$ such that $m(\{\omega \in (0,\8) :\; |x(\omega)| > s\})<\8$ for some $s$.  The K\"{o}the dual of $L^{\Phi}(0,\8)$ is defined by
\be
L^{\Phi^\times}(0,\8)=\bigg \{f \in  L_0(0, \8):\; \sup_{\|g\|_\Psi \le 1} \int_0^\8 |f(t) g(t)| dt < \infty \bigg \}.
\ee
with the norm
\be
\|f\|_{\Phi^\times}=\sup_{\|g\|_\Psi \le 1} \int_0^\8 |f(t)g(t)| dt.
\ee
Then
\beq\label{eq:equal-norm-orlicz}
\|f\|_{\Phi}\le\|f\|_{\Phi^\times}\le2\|f\|_{\Phi},\qquad f\in L^{\Phi}(0,\8).
\eeq

It is clear that
\beq\label{eq:orlicz-dual}
L^{\Phi}(0,\8)^*=L^{\Psi^\times}(0,\8),\qquad L^{\Psi^\times}(0,\8)^*=L^{\Phi}(0,\8).
\eeq
We refer to \cite{RR} for the details on Orlicz spaces.

Let   $E=L^{\Phi}(0,\8)$ and  $(\R,\nu)$ be a semifinite  von Neumann algebra. Then   $E(\R)=L^{(\Phi)}(\R)$ and $E^\times(\R)=L^{\Phi^\times}(\R)$ (see \cite[Proposition 5.3]{DDP2}).
Since  $L^{\Phi^\times}(0,\8)$  and  $L^{\Phi}(0,\8)$ are separable,
\be
L^{\Phi}(\R)^*=L^{\Psi^\times}(\R),\qquad L^{\Psi^\times}(\R)^*=L^{\Phi}(\R)
\ee
(see \cite[Theorem 5.6 and p. 745]{DDP2}). By \eqref{eq:equal-norm-orlicz}, we get
\beq\label{eq:nc-orlicz-dual}
L^{\Phi}(\R)^*=L^{\Psi}(\R)\qquad \mbox{isometrically}.
\eeq

For any $x\in L_0(\R)$,  set $\tilde{\mu}_t(x)=\frac{1}{t}\int_0^t\mu_s(x)ds$. Then $ \mu_t(x)\le\tilde{\mu}_t(x)$ for all $t>0$ and the map $x\mapsto\tilde{\mu}(x)$ is a sublinear operator from
$L_0(\R)$ to $L_0(0,\8)$.

We have that if  $1< a_{\Phi} \le b_{\Phi}<\8$, then there exists a constant $C>0$ such that
\beq\label{eq:mu-t-mu*-t-1}
\sup_{t>0}t \Phi \big [ \tilde{\mu}_t (x) \big ] \leq C \sup_{t>0} t \Phi \big [ \mu_t (x) \big ]
\eeq
for all $x \in L^{\Phi,\8}(\R)$. Consequently,
\beq\label{eq:mu-t-mu^*-t-2}
\sup_{t>0}\varphi_\Phi(t)\tilde{\mu}_t (x) \le C' \sup_{t>0}\varphi_\Phi(t) \mu_t (x),\qquad \forall x \in L^{\Phi,\8}(\R).
\eeq

We use \eqref{eq:inverses} and \eqref{eq:mu-t-mu^*-t-2} to obtain the following result.
If $1< a_{\Phi} \le b_{\Phi}<\8$, and  set
 \be
\|x\|_{\tilde{\Phi},\8}=\sup_{t>0}\frac{1}{\varphi_\Psi(t)} \int_0^t\mu_s(x)ds,\qquad \forall x \in L^{\Phi,\8}(\R),
\ee
then $\|x\|_{\tilde{\Phi},\8}$ is an equivalent norm on $L^{\Phi,\8}(\R)$.

For more details see Proposition 2.1 and Corollary 2.1 in \cite{BM1}.

 Let  $\M_1$ be a von Neumann
 subalgebra of $\M$ such that $\M_1$  is invariant under $\sigma^\varphi$. Then there is a (unique)
normal conditional expectation  $\E$ such that $\varphi\circ\E=\varphi$ (cf. \cite{T1}). Recall that the conditional expectation $\E $ extends to a contractive projection  from $L^p(\M )$ onto $L^p(\M_1)\;(1\le p<\8)$, the extension
will be denoted still by $\E$. If $1\le p,q\le\8$ and $\frac{1}{p}+\frac{1}{q}=1$, then
\begin{equation}\label{eq:conditional expectation}
  tr(\E(x)y)=tr(x\E(y)),\qquad \forall x\in L^p(\M),\;\forall y\in L^q(\M).
\end{equation}

\begin{proposition}\label{pro:conditional-expectation}
  If $1< a_{\Phi} \le b_{\Phi}<\8$, then   $\E$   extends to a  contractive projection from  $L^{\Phi,\8}(\N)$ onto $L^{\Phi,\8}(\N_1)$, where $\phi=\varphi|_{\M_1}$,  $\N_1=\M_1\rtimes_{\sigma^\phi}\mathbb{R}$ and $\N=\mathcal{M}\rtimes_{\sigma^\varphi}\mathbb{R}$.  $\check{\E}$ has the following properties:
 \begin{enumerate}[\rm(i)]
   \item If $0\le\alpha\le1$ and $D$ is positive operator in $L^1(\N_1)$, then for any $x\in\M$,
   \be
   \check{\E}(\Phi^{-1}(D)^\alpha x\Phi^{-1}(D)^{1-\alpha})=\Phi^{-1}(D)^\alpha \check{\E}(x)\Phi^{-1}(D)^{1-\alpha}.
   \ee
   \item If $\tilde{\Psi}$ is an N function such that $\tilde{\Psi}^{-1}(t)\Phi^{-1}(t)=t$ for any $t\ge0$, then for any positive operator $D$  in $L^1(\N_1)$ and $x\in L^{\Phi,\8}(\N)$,
    \be
   \check{\E}(\tilde{\Psi}^{-1}(D)^\alpha x\tilde{\Psi}^{-1}(D)^{1-\alpha})=\tilde{\Psi}^{-1}(D)^\alpha \check{\E}(x)\tilde{\Psi}^{-1}(D)^{1-\alpha}.
   \ee
  \end{enumerate}
\end{proposition}
\begin{proof}
The conditional expectation  $\E:\M\rightarrow\M_1$ commutes with $\{\sigma_{t}^\varphi\}_{t\in\real}$ (cf. \cite{C}), i.e.,
\be
\E\circ\sigma_{t}^\varphi=\sigma_{t}^\varphi\circ\E, \qquad \forall t\in\real.
\ee
 Using \cite[Theorem 4.1 (iv)]{H2}, we obtain that  the
conditional expectation $\E$ extends to a normal faithful conditional expectation $\check{\E}$
from $\N=\mathcal{M}\rtimes_{\sigma^\varphi}\mathbb{R}$ onto $\N_1=\M_1\rtimes_{\sigma^\phi}\mathbb{R}$, satisfying  $\tau_1\circ\check{\E}=\tau$
and $\widehat{\phi}\circ\check{\E}=\widehat{\varphi}$, where
 $\tau_1,\;\tau,\;\varphi,\;\widehat{\varphi}$ and $\widehat{\phi}$ are as in the Section 4. By \cite[Corollary 2.3]{BM1}, we may assume that $E=L^{\Phi,\8}(0,\8)$ is a symmetric Banach function space (with equivalent norm). Since $1< p_E \le q_E<\8$  (see \cite[page 26-28]{M} or \cite[Corollary 4.6]{BCLJ}), we can extend  to a contractive projection from
$L^{\Phi,\8}(\N)$ onto $L^{\Phi,\8}(\N_1)$. The extension will be denoted by $\check{\E}$.

Next, we  prove only (i). The proof of (ii) is
similar. Since $\Phi^{-1}(D)^\alpha\in L^{\Phi^{(\frac{1}{\alpha})},\8}_0(\N_1)$ and $\Phi^{-1}(D)^{1-\alpha}\in L^{\Phi^{(\frac{1}{1-\alpha})},\8}_0(\N_1)$, there are sequences $(y_n)_{n\ge1}$ and $(z_n)_{n\ge1}$  in $\N_1$ such that $y_n\rightarrow \Phi^{-1}(D)^\alpha$ in the norm of $L^{\Phi^{(\frac{1}{\alpha})},\8}(\N_1)$ and $z_n\rightarrow \Phi^{-1}(D)^{1-\alpha}$ in the norm of $L^{\Phi^{(\frac{1}{1-\alpha})},\8}(\N_1)$. It is clear that $y_nx\rightarrow \Phi^{-1}(D)^\alpha x$ in the norm of $L^{\Phi^{(\frac{1}{\alpha})},\8}(\N)$. Using Lemma \ref{lem:holder-weak}, we deduce that
\be
y_nxz_n\rightarrow \Phi^{-1}(D)^\alpha \Phi^{-1}(D)^\alpha x\Phi^{-1}(D)^{1-\alpha}
\ee
 in the norm of $L^{\Phi,\8}(\N)$. On the other hand,  we can extend $\check{\E}$ to a contractive projection from
$L^{\Phi^{(\alpha)},\8}(\N)$ (respectively, $L^{\Phi^{(\alpha)},\8}(\N)$) onto $L^{\Phi^{(\alpha)},\8}(\N_1)$ (respectively, $L^{\Phi^{(\alpha)},\8}(\N_1)$). We still it denoted by $\check{\E}$. Hence, $y_n=\check{\E}(y_n)\rightarrow\check{\E}(\Phi^{-1}(D)^\alpha)$ in the norm of $L^{\Phi^{(\frac{1}{\alpha})},\8}(\N_1)$ and $z_n=\check{\E}(z_n)\rightarrow \check{\E}(\Phi^{-1}(D)^{1-\alpha})$ in the norm of $L^{\Phi^{(\frac{1}{1-\alpha})},\8}(\N_1)$. Thus, $\check{\E}(\Phi^{-1}(D)^{\alpha})= \Phi^{-1}(D)^{\alpha}$ and $\check{\E}(\Phi^{-1}(D)^{1-\alpha})=\Phi^{-1}(D)^{1-\alpha}$. Therefore,
 \be
 \check{\E}(\Phi^{-1}(D)^\alpha x\Phi^{-1}(D)^{1-\alpha})&=&\lim_{n\rightarrow\8} \check{\E}(y_nxz_n)
=\lim_{n\rightarrow\8}y_n\check{\E}(x)z_n \\
&=&\Phi^{-1}(D)^\alpha \check{\E}(x)\Phi^{-1}(D)^{1-\alpha}.
 \ee
\end{proof}

For $0\le\alpha\le1$ we define
\be
\E_{\alpha}:\Phi^{-1}(D)^{\alpha}\M\Phi^{-1}(D)^{1-\alpha}\rightarrow\Phi^{-1}(D)^{\alpha}\M_1\Phi^{-1}(D)^{1-\alpha}
\ee
by
\be
\E_{\alpha}(\Phi^{-1}(D)^{\alpha}x\Phi^{-1}(D)^{1-\alpha})=\Phi^{-1}(D)^{\alpha}\E(x)\Phi^{-1}(D)^{1-\alpha},\qquad \forall x\in\M.
\ee

\begin{theorem}\label{thm:conditional-hagerup}
 Let $1< a_{\Phi} \le b_{\Phi}<\8$ and $0\le\alpha\le1$. Then $\E_{\alpha}$ extends to a  contractive projection from  $L^{\Phi,\alpha}(\M)$ onto $L^{\Phi,\alpha}(\M_1)$. The extension will be denoted still by $\E$.
\end{theorem}
\begin{proof} Let $\check{\E}$ be the extension of $\E$ in Proposition \ref{pro:conditional-expectation}. Recall that
 the Radon-Nikodym derivative of $\hat{\phi}$ with respect to $\tau_1$ is equal to $D$,
the Radon-Nikodym derivative of $\hat{\varphi}$ with respect to $\tau$. Applying  Lemma \ref{lem:density} and \eqref{eq:haagerub-subspace-weakorlicz}, we obtain that  for all $x\in\M$,
\be
\check{\E}(\Phi^{-1}(D)^{\alpha}x\Phi^{-1}(D)^{1-\alpha})=\Phi^{-1}(D)^{\alpha}\check{\E}(x)\Phi^{-1}(D)^{1-\alpha}=\Phi^{-1}(D)^{\alpha}\E(x)\Phi^{-1}(D)^{1-\alpha}.
\ee
Consequently, $\check{\E}|_{L^{\Phi,\alpha}(\M)}$ is the desired extension of $\E_{\alpha}$.
\end{proof}

 \begin{corollary}\label{cor:Haagerup's reduction} Let $\U,\;\U_n,\;\F_n\; (\forall n\in \mathbb{N})$ be fixed as in the above section. If $1< a_{\Phi} \le b_{\Phi}<\8$ and $0\le\alpha\le1$, then
 $L^{\Phi,\alpha}(\M)$  and all   $L^{\Phi,\alpha}(\U_n)$
are 1-complemented in $L^{\Phi,\alpha}(\U)$.
 \end{corollary}

\begin{theorem}\label{thm:equivalent-hagerup}
Let $1< a_{\Phi} \le b_{\Phi}<\8$ and   $0\le\alpha\le1$. Then   $L^{\Phi,\alpha}(\M,\varphi)$ is isomorphic to $L^{\Phi,1}(\M,\varphi)$.
\end{theorem}
\begin{proof}  Let $\F:\U\rightarrow\M$ be the
 conditional expectation in   \eqref{eq:equation-dual weight-condition}.  For $0\le\alpha\le1$ we define
\be
\F_{\alpha}:\Phi^{-1}(D)^{\alpha}\U\Phi^{-1}(D)^{1-\alpha}\rightarrow\Phi^{-1}(D)^{\alpha}\M\Phi^{-1}(D)^{1-\alpha}
\ee
by
\be
\F_{\alpha}(\Phi^{-1}(D)^{\alpha}x\Phi^{-1}(D)^{1-\alpha})=\Phi^{-1}(D)^{\alpha}\F(x)\Phi^{-1}(D)^{1-\alpha},\qquad \forall x\in\U.
\ee
By Theorem \ref{thm:conditional-hagerup},  $\F_{\alpha}$ extends to a  contractive projection from  $L^{\Phi,\alpha}(\U)$ onto $L^{\Phi,\alpha}(\M)$. It will be denoted still by $\F$.
Let $T_\alpha\;(0\le\alpha\le1)$ and $T$ be as in the proof of Lemma \ref{lem:isometric}. Then
\be
\F_{\alpha}(T_\alpha(x))=T_\alpha(\F_{\alpha}(x))\qquad\mbox{and}\qquad\F_1(T_1(x))=T_1(\F_1(x)),\qquad x\in\U.
\ee
Hence, we have that $\F\circ T=T\circ\F $.
It follows that
\be
 T(L^{\Phi,\alpha}(\M))=T(\F(L^{\Phi,\alpha}(\U)))=\F(T(L^{\Phi,\alpha}(\U)))=\F(L^{\Phi,1}(\U))= L^{\Phi,1}(\M)
\ee
Therefore, $ T|_{L^{\Phi,\alpha}(\M)}:\;L^{\Phi,\alpha}(\M)\rightarrow L^{\Phi,1}(\M)$ is a isometric isomorphisim.
This gives the desired result.
\end{proof}

 We will denote by
$L^{\Phi}(\M,\varphi)$ for the  Haagerup noncommutative Orlicz space
$L^{\Phi,\alpha}(\M,\varphi)$.

\begin{lemma}\label{lem:equivalent-Orlicz-function} Let
\be\tilde{\varphi}_\Phi(t)=\left\{\begin{array}{cc}
                         \frac{t}{\varphi_\Psi(t)} &\mbox{if}\; t>0 \\
                         0 &\mbox{if}\; t=0
                       \end{array}
\right.,\qquad
\tilde{\Phi}(t)=\left\{\begin{array}{cc}
                         \frac{1}{\tilde{\varphi}_\Phi^{-1}(\frac{1}{t})} &\mbox{if}\; t>0 \\
                         0 &\mbox{if}\; t=0
                       \end{array}
\right.,
\ee
then $\tilde{\Phi}(t)$, $t\tilde{\Phi}'(t)$ are growth functions, and $\tilde{\Phi}^{-1}(t)\Psi^{-1}(t)=t$ for any $t>0$,.
 \end{lemma}
\begin{proof} It is clear that $\tilde{\Phi}(t)$ is a growth function. By \eqref{eq:inverses},
\be
\begin{array}{l}
\tilde{\varphi}_\Phi(t)=t\Psi^{-1}(\frac{1}{t})=\frac{\Psi^{-1}(\frac{1}{t})}{\frac{1}{t}}\le\frac{2}{\Phi^{-1}(\frac{1}{t})}=2\varphi_\Phi(t),\\
\tilde{\varphi}_\Phi(t)=t\Psi^{-1}(\frac{1}{t})=\frac{\Psi^{-1}(\frac{1}{t})}{\frac{1}{t}}\ge\frac{1}{\Phi^{-1}(\frac{1}{t})}=\varphi_\Phi(t),\qquad \forall t>0.
\end{array}
\ee
Hence,
\be
\begin{array}{l}
\varphi_\Phi^{-1}(t)=\sup\{s:\;\varphi_\Phi(s)\le t\}\ge\sup\{s:\;\tilde{\varphi}_\Phi(s)\le t\}=\tilde{\varphi}_\Phi^{-1}(t),\\
\tilde{\varphi}_\Phi^{-1}(t)=\sup\{s:\;\tilde{\varphi}_\Phi(s)\le t\}\ge\sup\{s:\;\varphi_\Phi(s)\le \frac{t}{2}\}=\varphi_\Phi^{-1}(\frac{t}{2}),\qquad \forall t>0.
\end{array}
\ee
Therefore,
\be
\frac{1}{\varphi_\Phi^{-1}(\frac{1}{t})}\le\frac{1}{\tilde{\varphi}_\Phi^{-1}(\frac{1}{t})}\le\frac{1}{\varphi_\Phi^{-1}(\frac{1}{2t})} ,\qquad \forall t>0.
\ee
\beq\label{eq:equivalent-orlicz-functions}
\Phi(t)\le\tilde{\Phi}(t)\le\Phi(2t),\qquad \forall t>0.
\eeq
On the other hand,  any N-function in $\Delta_2$ is automatically  also in $\Delta_2\cap\Delta_\frac{1}{2}$. Using \eqref{eq:equivalent-orlicz-functions} and \cite[Proposition 1.4]{AB}, we get $\tilde{\Phi}(t)\in\Delta_2\cap\Delta_\frac{1}{2}$. Hence, $\inf_{t>0} \frac{t \tilde{\Phi}'(t)}{\tilde{\Phi} (t)}=a_{\tilde{\Phi}}>0$, and so $t \tilde{\Phi}'(t)\ge a_{\tilde{\Phi}}\tilde{\Phi} (t)$  for any $t>0$. It follows that $t\tilde{\Phi}'(t)$ is a growth function.
For any $t>0$, we have that
\beq\label{eq:inverse-equivalent-function}
\tilde{\Phi}^{-1}(t)\Psi^{-1}(t)=\frac{1}{\tilde{\varphi}_\Phi(\frac{1}{t})}\Psi^{-1}(t)=\frac{1}{\frac{1}{t}\Psi^{-1}(t)}\Psi^{-1}(t)=t.
\eeq
\end{proof}

Let $\M,\; \varphi, \;\varphi',\N,\;\N',\;\U,\;\U'$ and $\theta$ be as in the Section 4.

 \begin{theorem}\label{thm:independent of state-M} Let $1< a_{\Phi} \le b_{\Phi}<\8$.
Then
 $L^{\Phi}(\M,\varphi)$ is isometrically isomorphic to $L^{\Phi}(\M,\varphi')$.
\end{theorem}
\begin{proof}  Let $\tilde{\Psi}$ be defined as in Lemma \ref{lem:equivalent-Orlicz-function}. Suppose  $0\le\alpha\le1$.
If $x\in L^{\Phi,\alpha}(\M,\varphi)$, then from the proof of Theorem \ref{thm:independent of state}, we deduce that
$\theta(x)\in L^{\Phi,\alpha}(\U',\hat{\varphi'})$. On the other hand,  there is a sequence $(x_n)_{n\ge1}$ in $\M$ such that $\Phi^{-1}(D)^\alpha x_n\Phi^{-1}(D)^{1-\alpha}\rightarrow x$ in the norm of $L^{\Phi,\alpha}(\M,\varphi)$ and
\be
D^\alpha x_nD^{1-\alpha}&=&\tilde{\Psi}^{-1}(D)^\alpha \Phi^{-1}(D)^\alpha x_n\Phi^{-1}(D)^{1-\alpha}\tilde{\Psi}^{-1}(D)^{1-\alpha}\\
& \rightarrow& \tilde{\Psi}^{-1}(D)^\alpha x\tilde{\Psi}^{-1}(D)^{1-\alpha}
\ee
in the norm of $L^{1}(\M,\varphi)$, and so $\tilde{\Psi}^{-1}(D)^\alpha x\tilde{\Psi}^{-1}(D)^{1-\alpha}\in L^{1}(\M,\varphi)$. Hence,
\be
\tilde{\Psi}^{-1}(\theta(D))^\alpha \theta(x)\tilde{\Psi}^{-1}(\theta(D))^{1-\alpha}\in L^{1}(\M,\varphi').
\ee
 Let
 $\F'$ be the conditional expectation from
$\U'$ onto $\M$ determined by
\be
\F'(\lambda(t)x)=\left\{\begin{array}{rl}
                           x & \mbox{if}\; t=0 \\
                           0 & \mbox{otherwise}
                         \end{array}
\right.,\qquad \forall x\in\M,\qquad\forall t\in G.
\ee
Then it satisfies
\be
\widehat{\varphi'}\circ\F' = \widehat{\varphi'},\qquad\sigma_{t}^{\widehat{\varphi'}}\circ\F'=\F'\circ\sigma_{t}^{\widehat{\varphi'}}, \qquad\forall t\in \real.
\ee
For $0\le\alpha\le1$, let
\be
\F'_{\alpha}:\Phi^{-1}(D')^{\alpha}\U\Phi^{-1}(D')^{1-\alpha}\rightarrow\Phi^{-1}(D')^{\alpha}\M\Phi^{-1}(D')^{1-\alpha}
\ee
be defined by
\be
\F'_{\alpha}(\Phi^{-1}(D')^{\alpha}x\Phi^{-1}(D')^{1-\alpha})=\Phi^{-1}(D')^{\alpha}\F(x)\Phi^{-1}(D')^{1-\alpha},\qquad \forall x\in\U.
\ee
Using Theorem \ref{thm:conditional-hagerup}, we get that  $\F_{\alpha}$ can to extend to a  contractive projection from  $L^{\Phi,\alpha}(\U',\hat{\varphi'})$ onto $L^{\Phi,\alpha}(\M,\varphi')$. It will be denoted still by $\F'$. Therefore, by Proposition \ref{pro:conditional-expectation},
\be
\tilde{\Psi}^{-1}(\theta(D))^\alpha \theta(x)\tilde{\Psi}^{-1}(\theta(D))^{1-\alpha}&=&\F'(\tilde{\Psi}^{-1}(\theta(D))^\alpha \theta(x)\tilde{\Psi}^{-1}(\theta(D))^{1-\alpha})\\
&=&\tilde{\Psi}^{-1}(\theta(D))^\alpha \F'(\theta(x))\tilde{\Psi}^{-1}(\theta(D))^{1-\alpha}.
\ee
Since $\theta(D)$ is an invertible positive self-adjoint operator, we get
$\F'(\theta(x))=\theta(x)$, i.e,  $\theta(L^{\Phi}(\M,\varphi))\subseteq L^{\Phi}(\M,\varphi')$. Similarly we see that
$\theta^{-1}(L^{\Phi}(\M,\varphi'))\subseteq L^{\Phi}(\M,\varphi)$. Thus $\theta^{-1}(L^{\Phi}(\M,\varphi'))= L^{\Phi}(\M,\varphi)$ and we obtain the desired result.
\end{proof}

 In the sequel we will denote $L^{\Phi}(\M,\varphi)$ simply by $L^{\Phi}(\M)$.

Let $x,y\in\M$. By \eqref{eq:inverse-equivalent-function}, we have that $x\tilde{\Phi}^{-1}(D)\Psi^{-1}(D)y\in L^1(\M)$. Hence, using \eqref{eq:trace-norm}, \eqref{eq:generalized singular-haagerup},  \cite[Lemma 2.5(vii)]{FK}, \eqref{eq:inverse-equivalent-function}
and \eqref{eq:DistrSingularEquiv}, we get that
\be
\begin{array}{rl}
|\tr(x\tilde{\Phi}^{-1}(D)\Psi^{-1}(D)y)|&\le\|x\tilde{\Phi}^{-1}(D)\Psi^{-1}(D)y\|_1\\
&=2t\mu_{2t}(x\tilde{\Phi}^{-1}(D)\Psi^{-1}(D)y)\\
&\le2t\mu_{t}(x\tilde{\Phi}^{-1}(D))\mu_{t}(\Psi^{-1}(D)y)\\
&=2\frac{1}{\tilde{\Phi}^{-1}(\frac{1}{t})} \mu_t (x\tilde{\Phi}^{-1}(D))\frac{1}{\Psi^{-1}(\frac{1}{t})} \mu_t (\Psi^{-1}(D)y)\\
&\le2\|x\|_{\tilde{\Phi},0}\|y\|_{\Psi,1}.
\end{array}
\ee
 Therefore, we can extend $\tr$ on $L^{\tilde{\Phi},0}(\M)L^{\Psi,1}(\M)$ and denote it also $tr$, i.e., the bilinear form $(x,y)\mapsto
\tr(xy)$ defines a  bilinear continuous functional on  $L^{\tilde{\Phi},0}(\M)\times L^{\Psi,1}(\M)$ and
\beq\label{eq:holder-haagerup-orlicz}
|\tr(xy)|\le 2\|x\|_{\tilde{\Phi},0}\|y\|_{\Psi,1},\qquad x\in L^{\tilde{\Phi},0}(\M),\;y\in L^{\Psi,1}(\M).
\eeq

\begin{lemma}\label{lem:dual-un-Haagerup-Orlicz} Let $\U_n,\;\tau_n,\;\F_n\; (\forall n\in \mathbb{N})$ be fixed as in the above section. If  $1< a_{\Phi} \le b_{\Phi}<\8$,
then the bilinear form $(x, y) \mapsto \tr(xy)$ defines a duality bracket between $L^{\Phi,0}(\U_n)$  and $L^{\Psi,1}(\U_n)$, for which
$L^{\Psi.1}(\U_n)$  coincides (isomorphically) with the dual of $L^{\Phi,0}(\U_n)$.
\end{lemma}
\begin{proof} Let $g_{\varphi}\in L_0(\U_n,\tau_n)^+$ be the unique operator such that $\varphi(x)=\tau_n( xg_{\varphi})$ for any $x\in\U_n$. Let $\tilde{\Phi}$ be defined as Lemma \ref{lem:equivalent-Orlicz-function}. Using \eqref{eq:equivalent-orlicz-functions}, we get that $L^\Phi(\U_n)=L^{\tilde{\Phi}}(\U_n)$
and
\beq\label{eq:equivalent-orlicz-norm}
\|x\|_\Phi\le \|x\|_{\tilde{\Phi}}\le 2\|x\|_\Phi,\qquad \forall x\in  L^\Phi(\U_n).
\eeq
By Lemma \ref{lem:density-orlicz}, we know that $[\tilde{\Phi}^{-1}(g_{\varphi})^{\alpha}\U_n\tilde{\Phi}^{-1}(g_{\varphi})^{1-\alpha}]_{\tilde{\Phi}}=L^{\tilde{\Phi}}(\U_n)$ for any $0\le\alpha\le1$. Applying \eqref{eq:tr-trace-connection}, we deduce that
\beq\label{eq:dual-equivalence-tr}
\tr(xD)=\varphi(x)=\tau_n(xg_{\varphi}),\qquad\forall x\in \U_n.
\eeq
On the other hand, by \eqref{eq:inverse-equivalent-function}, we have that $\tilde{\Phi}^{-1}(D)\Psi^{-1}(D)=D$ and $\tilde{\Phi}^{-1}(g_{\varphi})\Psi^{-1}(g_{\varphi})=g_{\varphi}$.
Hence, from \eqref{eq:nc-orlicz-dual}, \eqref{eq:equivalent-orlicz-norm}, \eqref{eq:dual-equivalence-tr} and Theorem \ref{thm:haagerup-finite-state} it follows that for any $x,y\in\U_n$,
\be
\begin{array}{rl}
|\tr(x\tilde{\Phi}^{-1}(D)\Psi^{-1}(D)y)|&=|\tr(xDy)|=|\tr(yxD)|\\
&=|\tau_n(yxg_{\varphi})|=|\tau_n(xg_{\varphi}y)|\\
&=|\tau_n(x\tilde{\Phi}^{-1}(g_{\varphi})\Psi^{-1}(g_{\varphi})y)|\\
&\le C\|x\tilde{\Phi}^{-1}(g_{\varphi})\|_\Phi\|\Psi^{-1}(g_{\varphi})y\|_\Psi\\
&\le C\|x\tilde{\Phi}^{-1}(g_{\varphi})\|_{\tilde{\Phi}}\|\Psi^{-1}(g_{\varphi})y\|_\Psi\\
&=C\|x\tilde{\Phi}^{-1}(D)\|_{\tilde{\Phi},0}\|\Psi^{-1}(D)y\|_{\Psi,1}.
\end{array}
\ee
Since $\U_n\tilde{\Phi}^{-1}(D)$ and $\Psi^{-1}(D)\U_n$ are dense in $L^{\tilde{\Phi},0}(\U_n)$ and $L^{\Psi,1}(\U_n)$ respectively, it follows that for $x\in L^{\tilde{\Phi},0}(\U_n)$ and $y\in L^{\Psi,1}(\U_n)$ there are sequences $(x_n)_{n\ge1}\subset\U_n\tilde{\Phi}^{-1}(D)$ and $(y_n)_{n\ge1}\subset\Psi^{-1}(D)\U_n$ such that $\|x_n -x\|_{\tilde{\Phi},0}\rightarrow0$ and $\|y_n -y\|_{\Psi,1}\rightarrow0$  as $n\rightarrow\8$. Let $\tr(xy)=\lim_{n\rightarrow\8}\tr(x_ny_n)$. Then $\tr(xy)$ is well-defined. Hence, the map  $(x, y) \mapsto \tr(xy)$ is a bilinear form on $L^{\tilde{\Phi},0}(\U_n)\times L^{\Psi,1}(\U_n)$  and satisfies
\be
|\tr(xy)|\le C\|x\|_{\tilde{\Phi},0}\|y\|_{\Psi,1},\qquad x\in L^{\tilde{\Phi},0}(\U_n),\qquad y\in L^{\Psi,1}(\U_n).
\ee
If $y\in L^{\Psi,1}(\U_n)$, we set $\ell_y(x)=\tr(xy)$ for all $x\in L^{\tilde{\Phi},0}(\U_n)$. Then  $\ell_y\in L^{\tilde{\Phi},0}(\U_n)^*$ and $\|\ell_y\|\le C\|y\|_{\Psi,1}$.

Conversely, let $\ell\in L^{\tilde{\Phi},0}(\U_n)^*$. Since $L^{\tilde{\Phi}}(\U_n)^*=L^{\Psi}(\U_n)$,  there exists an operator $y'\in L^{\Psi}(\U_n)$ such that $\ell(x)=\tau_n(xy')$ for every $x\in L^{\tilde{\Phi}}(\U_n)$. By Lemma \ref{lem:density-orlicz}, $\ell(x)=\tau_n(xy')$ for every $x\in L^{\tilde{\Phi},0}(\U_n)$. On the other hand, $\Psi^{-1}(g_\varphi)\U_n$ is  dense in $L^{\Psi}(\U_n)$. Hence, there is a sequence $(y_n)_{n\ge1}\subset\U_n$ such that
 $\|y_n\Psi^{-1}(g_\varphi)-y'\|_{\Psi}\rightarrow0$  as $n\rightarrow\8$.  Applying Theorem \ref{thm:haagerup-finite-state}, we obtain that
\be
\|(y_n-y_m)\Psi^{-1}(D)\|_{\Psi,1}=\|(y_n-y_m)\Psi^{-1}(g_\varphi)\|_\Psi\rightarrow 0,\qquad \mbox{as}\;n\rightarrow\8\;\mbox{and}\;m\rightarrow\8,
\ee
i.e., $(y_n\Psi^{-1}(D))_{n\ge1}$ is a Cauchy sequence in $L^{\Psi,1}(\U_n)$. Therefore, there is an operator $y\in L^{\Psi,1}(\U_n)$ such that $\|y_n\Psi^{-1}(D)-y\|_{\Psi}\rightarrow0$  as $n\rightarrow\8$. Using \eqref{eq:tr-trace-connection},
we obtain that for all $x\in L^{\tilde{\Phi},0}(\U_n)$,
\be
\ell_y(x)=\tr(xy)=\lim_{n\rightarrow\8}\tr(xy_n\Psi^{-1}(D))=\lim_{n\rightarrow\8}\tau(xy_n\Psi^{-1}(g_\varphi))=\tau(xy')=\ell(x).
\ee
Thus $L^{\tilde{\Phi},0}(\U_n)^*=L^{\Psi,1}(\U_n)$. Hence, from \eqref{eq:equivalent-orlicz-functions} it  follows  that $L^{\Phi,0}(\U_n)^*=L^{\Psi,1}(\U_n)$.
\end{proof}

\begin{theorem}\label{thm:dual-converge-Haagerup-Orlicz} Let $\U,\;\U_n,\;\F_n\; (\forall n\in \mathbb{N})$ be fixed as in the above section. If  $1< a_{\Phi} \le b_{\Phi}<\8$,
then   $L^{\Phi}(\U)^*=L^{\Psi}(\U)$ with equivalent norm.
 \end{theorem}
\begin{proof}  For $y\in L^{\Psi,1}(\U)$, we define
\be
\ell_y(x)=\tr(xy),\qquad \forall x\in L^{\tilde{\Phi},0}(\U).
\ee
Then from \eqref{eq:holder-haagerup-orlicz} it follows that $\ell_y\in L^{\tilde{\Phi},0}(\U)^*$ and $\|\ell_y\|\lesssim\|y\|_{\Psi}$.

Conversely, if $\ell\in L^{\tilde{\Phi},0}(\U)^*$, then  $\ell|_{L^{\tilde{\Phi},0}(\U_n)}\in L^{\tilde{\Phi},0}(\U_n)^*$. From the proof of Lemma \ref{lem:dual-un-Haagerup-Orlicz}, we know that there is an $y_n\in L^{\Psi,1}(\U_n)$ such that $\|y_n\|\approx\|\ell\|$  and  $\ell(x)=\tr(x y_n)$   for any $x\in L^{\tilde{\Phi},0}(\U_n)$. On the other hand, using Theorem \ref{thm:conditional-hagerup}, we get that $\F_n$  extends to a  contractive projection from    $L^{\Psi,1}(\U)$ (respectively,  $L^{\tilde{\Phi},0}(\U)$) onto $L^{\Psi,1}(\U_n)$ (respectively,  $L^{\tilde{\Phi},0}(\U_n)$). We denote the extension still by $\F_n$. By \eqref{eq:equation-dual weight-condition-n}, we deduce that $\F_n \circ\F_{n+1} = \F_n$  for any $n\in\mathbb{N}$. Let $x\in L^{\tilde{\Phi},0}(\U_n)$. Then
$\tr(x y_n)=\tr(x y_{n+1})$. Choose sequences $(x_k)_{k\ge1}\subset\U_n$ and $(z_k)_{k\ge1}\subset\U_{n+1}$ such that  $x_k\tilde{\Phi}^{-1}(D)$ converges to $x$ in the norm of $L^{\tilde{\Phi},0}(\U_n)$ and  $\Psi^{-1}(D)z_k$ converges to $y_{n+1}$ in the norm of $L^{\Psi,1}(\U_{n+1})$.
Then
\be\begin{array}{rl}
     \tr(x y_n) & =\tr(x y_{n+1})=\lim_{k\rightarrow\8}\tr(x_k\tilde{\Phi}^{-1}(D)\Psi^{-1}(D)z_k)\\
      &=\lim_{k\rightarrow\8}\tr(\F_n(x_k)\tilde{\Phi}^{-1}(D)\Psi^{-1}(D)z_k)\\
      &=\lim_{k\rightarrow\8}\tr(x_k\tilde{\Phi}^{-1}(D)\Psi^{-1}(D)\F_n(z_k))\\
      &= \tr(x \F_n(y_{n+1}).
   \end{array}
\ee
Thus gives $\F_n(y_{n+1}) = y_n$ for any $n\in\mathbb{N}$, since $x$ is arbitrary element in $L^{\tilde{\Phi},0}(\U_n)$. By Proposition \ref{pro:conditional-expectation},  $\F_n$   extends to a  contractive projection  $\widehat{\F_n}$ from  $L^{\Phi,\8}(\mathfrak{N})$ onto $L^{\Phi,\8}(\mathfrak{N}_n)$, where  $\mathfrak{N}=\U\rtimes_{\sigma^{\widehat{\varphi}}}\mathbb{R}$ and $\mathfrak{N}_n=\U_n\rtimes_{\sigma^{\widehat{\varphi}}}\mathbb{R}$. Hence, $\widehat{\F_n}(y_{n+1}) = y_n$ for any $n\in\mathbb{N}$. Thus
$y=(y_n)_{n\ge1}$ is a  $L^{\Psi,\8}_0(\mathfrak{N})$-bounded martingale. Therefore, there is
  an $y_\8\in L^{\Psi,\8}(\mathfrak{N})$ such that $y_n$ converges to $y_\8$ in the norm of $L^{\Psi,\8}(\mathfrak{N})$ and $\|y_\8\|_{\Phi,\8}\lesssim\|\ell\|$ (see Subsection 2.4). By  Theorem \ref{thm:Haagerup's reduction}, $y_\8\in L^{\Psi,1}(\U)$. It follows that
  \be
  \ell(x)=\tr(yx),\qquad x\in L^{\tilde{\Phi},0}(\U).
  \ee
By Theorem \ref{thm:equivalent-hagerup}, we obtain the desired result.
\end{proof}

Using Theorem \ref{thm:dual-converge-Haagerup-Orlicz} and  Corollary \ref{cor:Haagerup's reduction}, we get the following result.

\begin{theorem}\label{thm:dual-hagerup} Let $1< a_{\Phi} \le b_{\Phi}<\8$. Then  $L^{\Phi}(\M)^*=L^{\Psi}(\M) $ with equivalent norms.
\end{theorem}

\begin{corollary}\label{cor:reflexivel-hagerup} Let $1< a_{\Phi} \le b_{\Phi}<\8$. Then  $L^{\Phi}(\M)$ is reflexive.
\end{corollary}
\section{Applications}
In this section,  we always denote by  $\Phi$ an  N-function on $(0,\8)$ such that $1< a_{\Phi} \le b_{\Phi}<\8$  and by $\Psi$ its a complementary N-function.

Let $a=(a_{n})$ be a finite sequence in $L^{\Phi}(\M)$,
we define
\be
\|a\|_{L^{\Phi}(\M,\el_{c}^{2})}= \Big \| \Big ( \sum_n
|a_{n}|^{2} \Big )^{1/2} \Big \|_{L^{\Phi}(\M)},\qquad
\|a\|_{L^{\Phi}(\M,\el_{r}^{2})} = \Big \| \Big ( \sum_{n
} |a_{n}^{*}|^{2} \Big )^{1/2} \Big \|_{L^{\Phi}(\M)},
\ee
respectively. This gives two norms on the family of all finite sequences in
$L^{\Phi}(\M)$. Indeed, let $\B(\el^{2}) $ be the algebra of all bounded operators on
$\el^{2}$ with the usual trace $\Tr$. We consider the von Neumann
algebra tensor product $\M\otimes\B(\el^{2})$, equipped with the tensor
product weight  $\varphi \otimes \Tr$. It is clear that $\varphi \otimes \Tr$ is a faithful normal semi-finite weight, but  it is no longer a state.
Therefore, to continue the discussion in the case of states, we will consider $\B(\el^{2}_m)$  with an arbitrary positive integer $m$ instead of
$\B(\el^{2})$. We still denote $\B(\el^{2}_m)$  by $\B(\el^{2})$.   Let $\psi=\varphi \otimes \Tr$. Then the corresponding
one parameter automorphism group is
\be
\sigma_{t}^\psi = \sigma_{t}^\varphi\otimes id_{{\mathcal{B}}(\el^{2})},\qquad t \in \real.
\ee
Hence,
\be
[\M\otimes{\mathcal{B}}(\el^{2})]\rtimes_{\sigma^\psi}\real= (\mathcal{M}\rtimes_{\sigma^\varphi}\real)\otimes{ \mathcal{B}}(\el^{2})= \N\otimes{ \mathcal{B}}(\el^{2}).
\ee
The canonical faithful normal semifinite  trace
$\nu$ on $[\M\otimes{\mathcal{B}}(\el^{2})]\rtimes_{\sigma^\psi}\real$ is the tensor product $\tau\otimes\Tr$
(recalling that $\tau$ is the canonical trace on $\N$). Let $L^{\Phi}(\M\otimes{\mathcal{B}}(\el^{2}))$ denote
the Haagerup $L^\Phi$-space associated with $\tau\otimes\Tr$. Observe that the distinguished
tracial functional on $L^1(\M\otimes{\mathcal{B}}(\el^{2}))$ is equal $\tr\otimes\Tr$.

 Now, any finite sequence $a=(a_{n})_{n\ge1}$ in $L^{\Phi}(\M)$ can be regarded as an element in
$L^{\Phi}(\M\otimes {\mathcal{B}}(\el^{2}))$ via the
following map
\be
a\longmapsto T(a)= \left(
\begin{matrix}
a_{1} & 0 & \ldots \\
a_{2} & 0 & \ldots \\
\vdots & \vdots & \ddots
\end {matrix}
\right ),
\ee
that is, the matrix of $T(a)$ has all vanishing entries except those
in the first column which are the ${a_{n}}$'s. Such a matrix is
called a column matrix, and the closure in
$L^{\Phi}(\M\otimes {\mathcal{B}}(\el^{2}))$ of all column
matrices is called the column subspace of
$L^{\Phi}(\M\otimes {\mathcal{B}}(\el^{2}))$. Since
\be
\|a\|_{L^{\Phi}(\M,\el_{c}^{2})}=\||T(a)|\|_{L^{\Phi}(\M\otimes
{\mathcal{B}}(\el^{2}))}= \|T(a)\|_{L^{\Phi}(\M\otimes
{\mathcal{B}}(\el^{2}))},
\ee
then $\|.\|_{L^{\Phi}(\M,\el_{c}^{2})}$ defines a norm
on the family of all finite sequences of $L^{\Phi}(\M)$.
The corresponding completion $L^{\Phi}(\M,\el_{c}^{2})$ is a Banach space. It is clear that a sequence
$a=(a_{n})_{n\ge1}$ in $L^{\Phi}(\M)$ belongs to
$L^{\Phi}(\M, \el_{c}^{2})$ if and only if
\be
\sup_{n\geq1} \Big \| \Big ( \sum_{k=1
}^{n}|a_{k}|^{2} \Big )^{1/2} \Big \|_{L^{\Phi}(\M)}<\infty.
\ee
If this is the case, $\big ( \sum_{k=1
}^{\infty}|a_{k}|^{2} \big )^{1/2}$ can be appropriately defined as an
element of $L^{\Phi}(\M)$. Similarly, $
\|.\|_{L^{\Phi}(\M,\el_{r}^{2})}$ is also a norm on the family
of all finite sequence in $L^{\Phi}(\M)$, and the corresponding completion $L^{\Phi}(\M, \el_{r}^{2})$ is a Banach space, which is isometric to the row subspace of
$L^{\Phi}(\M\otimes {\mathcal{B}}(\el^{2}))$ consisting of
matrices whose nonzero entries lie only in the first row. Observe
that the column and row subspaces of $L^{\Phi}(\M\otimes
{\mathcal{B}}(\el^{2}))$ are complemented by
in \cite[Lemma 1.2]{JX} and \cite[Theorem 4.8]{D}.

We also need $L^\Phi_d(\M)$, the space of all sequences $a=(a_m)_{m\ge1}$ in $L^\Phi(\M)$ such that
$$
\|a\|_{L^{\Phi}_d({\mathcal{M}})}=\|diag(a_m)\|_{L^{\Phi}({\mathcal{M}}\otimes
{\mathcal{B}}(\ell_{2}))}<\8,
$$
where
$diag(a_m)$  is
the matrix with the $a_m$ on the diagonal  and zeroes
elsewhere.

\subsection{Noncommutative martingale inequalities}

The aim of this section is  using Theorem
\ref{thm:Haagerup's reduction} to extend some martingale inequalities in the tracial
case to  the nontracial case, i.e. to extend some results obtained in \cite{H2} in the case $L_p$-case to the Orlicz space case.

We keep all notations introduced in the preceding sections, we fix an increasing filtration $(\M_m)_{m\ge1}$ of von Neumann subalgebras of $\M$
whose union is $w^*$-dense in $\M$. Assume that for each $m\in\mathbb{N}$ there exists a normal
faithful conditional expectation $\E_m$ from $\M$ onto $\M_m$ such that
\beq\label{eq:equality-condition}
\varphi\circ\E_m=\varphi.
\eeq
 It is well known that such an $\E_m$ is unique and satisfies
\beq\label{eq:modular automorphism-invariant}
 \sigma_{t}^\varphi\circ\E_m=\E_m \circ\sigma_{t}^\varphi,\;\qquad \forall t\in\real.
\eeq
Recall that the existence of a conditional expectation $\E_m$ satisfying \eqref{eq:equality-condition} is equivalent
to the $\sigma_{t}^\varphi$-invariance of $\M_m$ (see \cite{T1}). The restriction of $\sigma_{t}^\varphi$ to $\M_m$ is the modular automorphism group of $\varphi|_{\M_m}$. We will not differentiate $\varphi$ and $\sigma_{t}^\varphi$ from their restrictions
to $\M_m\;(m\in\mathbb{N})$. We have that
\beq\label{eq:condition-expectations-eq}
\E_m \circ\E_k = \E_m \circ \E_k = \E_{\min(m,k)},\qquad \forall m,\; k\in \mathbb{N}.
\eeq
Let $G=\cup_{n\ge1}2^{-n}\mathbb{Z}$ (fixed throughout this sections). Set
\be
\U(\M) =\M\rtimes_{\sigma^\varphi} G,\qquad \U(\M_m) =\M_m\rtimes_{\sigma^\varphi} G,\qquad \forall m\in\mathbb{N}.
\ee
By \eqref{eq:condition-expectations-eq}, we consider $\U(\M_m)$   as a von Neumann subalgebra of
$\U(\M)$.
Let $\widehat{\varphi}$ be the dual weight of $\varphi$ on $\U(\M)$. Then $\widehat{\varphi}$ is again a faithful
normal state on $\U(\M)$ whose restriction on $\M$ is $\varphi$. We have that $\widehat{\varphi}|_{\U(\M_m)}$ is equal to the dual weight of $\varphi|_{\M_m}$.
 As usual, we
denote this restriction still by $\widehat{\varphi}$. The increasing sequence of the von Neumann
subalgebras of $\U(\M)$ constructed in Section 4 relative to $\M$ is denoted by
$(\U_n(\M))_{n\ge1}$, and that relative to $\M_m$ by $(\U_n(\M_m))_{n\ge1}$. All $\U_n(\M)$ and $\U_n(\M_n)$
are von Neumann subalgebras of $\U(\M)$. It is clear that
\beq
\U_n(\M_m) = \U_n(\M_m)\cap \U_n(\M),\qquad  n \in \mathbb{N}.
\eeq
Let $\F:\U(\M) \rightarrow\M$ be the conditional expectation defined by \eqref{conditional-expectation1}. Then $\F|_{\U(\M_n)}$  is the corresponding conditional expectation from $\U(\M_m)$ onto $\M_m$, its again denoted by $\F$. Let $\F_n: \U(\M)\rightarrow\U_n(\M)$ be the conditional expectation
constructed in Section 4. On the other hand,  $\E_m$ extends to a normal faithful conditional expectation  $\hat{\E}_m$  from $\U(\M)$ onto
$\U(\M_m)$ such that
\be
\hat{\varphi}\circ\hat{\E}_m=\hat{\varphi},\qquad \sigma_{t}^{\widehat{\varphi}}\circ\hat{\E}_m=\hat{\E}_m\circ\sigma_{t}^{\widehat{\varphi}}, \qquad\forall t\in \real.
\ee
Hence,
\beq\label{eq:condtional expectations-conditional expectations}
\hat{\E}_m\circ\F=\F\circ\hat{\E}_m,\qquad \hat{\E}_m\circ\F_n=\F_n\circ\hat{\E}_m,\qquad n,\;m \in \mathbb{N}.
\eeq
Therefore, $\U(\M_m)$ and $\U_n(\M)$ are respectively invariant under $\hat{\E}_m$ and $\F_n$ for
all $n,\;m \in \mathbb{N}$; moreover, $\F_n|_{\U(\M_m)}: \U(\M_m)\rightarrow \U_n(\M_m)$ and $\hat{\E}_m|_{\U_n(\M)}:\U_n(\M)\rightarrow\U_n(\M_m)$ are normal faithful conditional expectations. In this section, we will not distinguish
these conditional expectations and their respective restrictions. All these maps preserve the dual state $\hat{\varphi}$ (for more details see \cite[Section 6]{H2}).

Let $1< a_{\Phi} \le b_{\Phi}<\8$. By Proposition \ref{pro:Haagerup-subspace}, $L^\Phi(\M_m)$ is naturally identified as a subspace of
$L^\Phi(\M)\;(m \in \mathbb{N})$. Using Theorem \ref{thm:conditional-hagerup}, we obtain that $\E_m$ ($
\E_m(\Phi^{-1}(D)^{\frac{1}{2}}x\Phi^{-1}(D)^{\frac{1}{2}})=\Phi^{-1}(D)^{\frac{1}{2}}\E_m(x)\Phi^{-1}(D)^{\frac{1}{2}},\; \forall x\in\M$)
extends to a positive contractive projection from $L^\Phi(\M)$ onto $L^\Phi(\M_m)\;(m \in \mathbb{N})$. From the prof of Theorem \ref{thm:conditional-hagerup}, we know that  if  $1\le r,p,q<\infty$ such
that $r^{-1}+p^{-1}+q^{-1}\le1$, then
\beq\label{eq:modular-property}
\E_m(axb) = a\E_m(x)b,\qquad a \in L^{\Phi^{(r)}}(\M_m), \;b \in L^{\Phi^{(p)}}(\M_m ), \;x \in L^{\Phi^{(q)}}(\M).
\eeq
 $L^\Phi(\U_n(\M))$ and  $L^\Phi(\U_n(\M_m))$
 are naturally identified as subspaces of $L^\Phi(\U(\M))$. Thus, $L^\Phi(\U(\M))$ is the largest
space among all these noncommutative $L^\Phi$-spaces. Since $\F,\;\F_n$ and $\E_n$ preserve $\hat{\varphi}$
and commute with $\sigma_{t}^{\widehat{\varphi}}$, these conditional expectations extend to positive contractive
projections on $L^\Phi(\U(\M))$; moreover, their extensions satisfy
the equality \eqref{eq:modular-property}. Since all these conditional expectations commute, so
do their extensions. As usual, we use the same symbol to denote a map and its
extensions.

A non-commutative $L^{\Phi}$-martingale with respect to $(\M_{m})_{m\ge1}$ is a sequence $x=(x_{m})_{m\ge1}$
such that $x_{m} \in L^{\Phi}(\M_m)$ and
\be
\E_n(x_{m+1})=x_m, \qquad \forall m\in \mathbb{N}.
\ee
Let $\|x\|_{L^{\Phi}({\mathcal{M}})}=\sup_{m\ge1}\|x_{m}\|_{L^{\Phi}({\mathcal{M}})}$. If $\|x\|_{L^{\Phi}({\mathcal{M}})}
<\infty$, then $x$ is said to be a bounded $L^{\Phi}$-martingale.
The martingale difference sequence of $x$, denoted
by $dx=(dx_{m})_{m\ge1},$ is defined as
$$
dx_{1}=x_{1},\qquad dx_{m}=x_{m}-x_{m-1},\qquad m\ge2.
$$
 A martingale $x$ is called a finite martingale if there exists $N$ such that $dx_m = 0$
for all $m\ge N$
\begin{remark}\label{re:conver}
\begin{enumerate}[\rm (i)]

\item Let $x_{\infty}\in L^{\Phi}({\mathcal{M}})$. Set
$x_{m}={\mathcal{E}}_{m}(x_{\infty})$ for all $m\geq1$. Then
$x=(x_{m})$ is a bounded $L^{\Phi}$-martingale and
$\|x\|_{L^{\Phi}({\mathcal{M}})}=\|x_{\infty}\|_{L^{\Phi}({\mathcal{M}})}$.

\item Suppose $\Phi$ is an N-function with $1<a_{\Phi}\le
b_{\Phi}<\infty$. Then $L^{\Phi}(\M)$ is reflexive. By the standard argument we conclude that any bounded
noncommutative martingale $x=(x_{m})$ in $L^{\Phi}(\M)$
converges to some $x_{\infty}$ in $L^{\Phi}(\M)$ and
$x_{m}=\E_{m}(x_{\infty})$ for all $m\ge1$.
\end{enumerate}
\end{remark}
Let $x$ be a noncommutative martingale. Set
$$
S_{c,m}(x)= \Big ( \sum_{k= 1  }^{m}|dx_{k}|^{2} \Big )^{1/2} \qquad \mbox{and}
\qquad S_{r,m}(x)= \Big ( \sum_{k=1}^{m}|dx_{k}^{*}|^{2} \Big )^{1/2}.
$$
By the preceding discussion, $dx$ belongs to
$L^{\Phi}({\mathcal{M}},\el_{c}^{2})$ (resp.
$L^{\Phi}({\mathcal{M}}, \el_{r}^{2}))$ if and only if $(S_{c,m}(x))_{m\ge1}$
(resp. $(S_{r,m}(x))_{m\ge1}$) is a bounded sequence in
$L^{\Phi}({\mathcal{M}})$; in this case,
\be
S_{c}(x)= \Big ( \sum_{k= 1    }^{\infty}|dx_{k}|^{2} \Big )^{1/2} \qquad
\mbox{and} \qquad S_{r}(x)= \Big ( \sum_{k= 1
}^{\infty}|dx_{k}^{*}|^{2} \Big )^{1/2}
\ee
are elements in $L^{\Phi}({\mathcal{M}})$.

We define $\mathcal{H}_{c}^{\Phi}({\mathcal{M}})$ (resp. $\mathcal{H}_{r}^{\Phi}({\mathcal{M}})$)
to be the space of all $L^{\Phi}$-martingales with respect to
$({\mathcal{M}}_{m})_{m\ge1}$ such that $dx \in
L^{\Phi}({\mathcal{M}}, \el_{c}^{2})$ (resp. $dx \in
L^{\Phi}({\mathcal{M}}, \el_{r}^{2})$ ), equipped with the norm
\be
\|x\|_{\mathcal{H}_{c}^{\Phi}({\mathcal{M}})}=\|dx\|_{L^{\Phi}({\mathcal{M}}, \el_{c}^{2})
} \qquad \big ( \mbox{resp.} \;
\|x\|_{\mathcal{H}_{r}^{\Phi}({\mathcal{M}})}=\|dx\|_{L^{\Phi}({\mathcal{M}}, \el_{r}^{2}) } \big ).
\ee
$\mathcal{H}_{c}^{\Phi}({\mathcal{M}})$ and $\mathcal{H}_{r}^{\Phi}({\mathcal{M}})$ are
Banach spaces. Note that if $x\in \mathcal{H}_{c}^{\Phi}({\mathcal{M}}),$
\be
\|x\|_{\mathcal{H}_{c}^{\Phi}({\mathcal{M}})}=\sup_{m\ge1}\|S_{c,m}(x)\|_{L^{\Phi}({\mathcal{M}})}
=\|S_{c}(x)\|_{L^{\Phi}({\mathcal{M }})}.
\ee
The similar equalities hold for $\mathcal{H}_{r}^{\Phi}({\mathcal{M}}).$

Now, we define the Hardy spaces of noncommutative martingales as follows:
If $b_{\Phi} < 2,$
\be
\mathcal{H}^{\Phi}({\mathcal{M}}) = \mathcal{H}_{c}^{\Phi}({\mathcal{M}}) + \mathcal{H}_{r}^{\Phi}({\mathcal{M}}),
\ee
equipped with the norm
\be
\|x\|_{\mathcal{H}^{\Phi}}=\inf \{ \|y\|_{ \mathcal{H}_{c}^{\Phi}({\mathcal{M}})}+\|z\|_{\mathcal{H}_{r}^{\Phi}({\mathcal{M}})}:
x=y+z,\; y \in \mathcal{H}_{c}^{\Phi}({\mathcal{M}}),\; z \in
\mathcal{H}_{r}^{\Phi}({\mathcal{M}})\}.
\ee
If $2< a_{\Phi}$,
$$
\mathcal{H}^{\Phi}({\mathcal{M}}) = \mathcal{H}_{c}^{\Phi}({\mathcal{M}}) \cap
\mathcal{H}_{r}^{\Phi}({\mathcal{M}}),
$$
equipped with the norm
$$
\|x\|_{\mathcal{H}^{\Phi}}=\max \{ \|x\|_{\mathcal{H}_{c}^{\Phi}({\mathcal{M}})},\;
\|x\|_{\mathcal{H}_{r}^{\Phi}({\mathcal{M}})} \}.
$$
Let
\be
\N=\mathcal{M}\rtimes_{\sigma^\varphi}\real,\qquad \N_m=\M_m\rtimes_{\sigma^\phi}\real,\qquad m\in\mathbb{N}
\ee
and $\tau$ be the canonical  normal semi-finite faithful trace  on $\N$. It is clear that $\E_m$ extends to a  conditional expectation  $\check{\E}_m$  from $L^{\Phi,\8}({\mathcal{N}})$ onto
$L^{\Phi,\8}({\mathcal{N}}_m)$ such that
\beq\label{eq:condiional-cross-product}
\check{\E}_m \circ\check{\E}_k = \breve{\E}_m \circ \check{\E}_k = \check{\E}_{\min(m,k)},\qquad \forall m,\; k\in \mathbb{N}.
\eeq
We now consider the conditioned versions of square functions and Hardy spaces developed
in \cite{JX}. Let $a=(a_m)_{m\ge1}$ be a finite sequence in $L^1(\N)\cap\N$. We define
(recalling $\check{\E}_0=\check{\E}_1$)
$$
\|a\|_{L^{\Phi,\8}_{cond}(\M, \el_{c}^{2})}=\| \Big ( \sum_{n\ge1}\check{\E}_{n-1}(|a_{n}|^{2}) \Big )^{\frac{1}{2}}\|_{L^{\Phi,\8}({\mathcal{N}})}.
$$
It is shown in \cite{RW} that $\|\cdot\|_{L^{\Phi,\8}_{cond}(\N, \el_{c}^{2})}$ is a norm on the vector space of all finite
sequences in $L^1(\N)\cap\N$. We define $L^{\Phi,\8}_{cond}(\N, \el_{c}^{2})$ is the corresponding completion. Similarly, we
define the conditioned row space $L^{\Phi,\8}_{cond}(\N, \el_{r}^{2})$.  Note that $L^{\Phi,\8}_{cond}(\N, \el_{c}^{2})$ (resp. $L^{\Phi,\8}_{cond}(\N, \el_{r}^{2})$)
is the conditioned version of $L^{\Phi,\8}(\N, \el_{c}^{2})$ (resp. $L^{\Phi,\8}(\N, \el_{r}^{2})$). The space
$L^{\Phi,\8}_{cond}(\N, \el_{c}^{2})$ ( resp. $L^{\Phi,\8}_{cond}(\N, \el_{r}^{2})$) can be viewed as a closed
subspace of the column (resp. row) subspace of $L^{\Phi,\8}({\mathcal{N}}\otimes {\mathcal{B}}(\ell_{2}(\mathbb{N}^2)))$.
We refer
to \cite{J,JX,RW} for more details on this.

Notice that $\check{\E}_m|_{L^{\Phi}(\M)}=\E_m\;( m\in \mathbb{N})$. For any finite sequence $a=(a_{n})_{m\ge1}$ in $L^{\Phi}(\M)$,
we define
\be
\|a\|_{L^{\Phi}_{cond}(\M,\el_{c}^{2})}= \Big \| \Big ( \sum_{n\ge1}\E_{n-1}(
|a_{n}|^{2}) \Big )^{1/2} \Big \|_{L^{\Phi}(\M)}=\| \Big ( \sum_{n\ge1}\check{\E}_{n-1}(|a_{n}|^{2}) \Big )^{\frac{1}{2}}\|_{L^{\Phi,\8}({\mathcal{N}})},\\
\|a\|_{L^{\Phi}_{cond}(\M,\el_{r}^{2})} = \Big \| \Big ( \sum_{n\ge1
}\E_{n-1}( |a_{n}^{*}|^{2}) \Big )^{1/2} \Big \|_{L^{\Phi}(\M)}=\| \Big ( \sum_{n\ge1}\check{\E}_{n-1}(|a_{n}^*|^{2}) \Big )^{\frac{1}{2}}\|_{L^{\Phi,\8}({\mathcal{N}})},
\ee
respectively. Then this gives two norms on the family of all finite sequences in
$L^{\Phi}(\M)$. Let  $L^{\Phi}_{cond}(\M, \el_{c}^{2})$ and $L^{\Phi}_{cond}(\M, \el_{r}^{2})$ be the corresponding completions.  Then
$L^{\Phi}_{cond}(\M, \el_{c}^{2})$ ( resp. $L^{\Phi}_{cond}(\M, \el_{r}^{2})$)  can  be viewed
as a subspace of $L^{\Phi}({\mathcal{M}}\otimes {\mathcal{B}}(\ell_{2}(\mathbb{N}^2)))$ as column (resp. row) vectors.

For a finite noncommutative $L^\Phi$-martingale  $x=(x_{m})_{m\geq 1}$ define (with $\E_{0}=\E_1$)
$$
\|x\|_{\h_c^{\Phi}({\mathcal{M}})}=\|dx\|_{L^{\Phi}_{cond}(\M, \el_{c}^{2})}\qquad \mbox{and}
\qquad
 \|x\|_{\h_r^{\Phi}({\mathcal{M}})}= \|dx\|_{L^{\Phi}_{cond}(\M, \el_{r}^{2})}.
$$
Let $\h_c^{\Phi}({\mathcal{M}})$ and $\h_r^{\Phi}({\mathcal{M}})$ be the corresponding completions. Then $\h_c^{\Phi}({\mathcal{M}})$ and $\h_r^{\Phi}({\mathcal{M}})$ are Banach
spaces. We define the column and row conditioned square functions as follows. For any finite
martingale $x = (x_m)_{m\ge1}$ in $L^\Phi(M)$, we set
$$
s^c (x)= \Big ( \sum_{k\ge1 }\E_{k-1}(|dx_{k}|^{2}) \Big )^{\frac{1}{2}} \qquad \mbox{and}
\qquad s^r (x)= \Big ( \sum_{k\ge1}\E_{k-1}(|dx_{k}^{*}|^{2}) \Big )^{\frac{1}{2}}.
$$
Then
$$
\|x\|_{\h_c^{\Phi}({\mathcal{M}})} = \|s^c (x)\|_{L^{\Phi}({\mathcal{M }})} \qquad \mbox{and}
\qquad \|x\|_{h_r^{\Phi}({\mathcal{M}})} = \|s^r (x)\|_{L^{\Phi}({\mathcal{M }})}.
$$
Let $x=(x_{m})_{m\ge1}$ be a finite $L^{\Phi}$-martingale, we set
$$
s^d (x)=diag(|dx_n|).
$$
We note that
$$
\|s^d (x)\|_{L^{\Phi}({\mathcal{M}}\otimes
{\mathcal{B}}(\ell_{2}))}=\|dx_n\|_{L^{\Phi}_d({\mathcal{M}})}
$$
Let $\h_d^\Phi(\M)$ be the subspace of $L^{\Phi}_d({\mathcal{M}})$ consisting of all martingale difference sequences.

Recall that the mixture conditioned $\mathcal{H}^\Phi$-spaces of noncommutative martingales are defined as follows: For $2< a_{\Phi}$, $\h^\Phi (\M) = \h_c^\Phi (\M) \cap \h_r^\Phi (\M) \cap \h_d^\Phi (\M)$ equipped with the norm $\| x \|_{\h^\Phi} = \| x\|_{\h_c^\Phi} + \|x\|_{\h_r^\Phi} + \|x\|_{\h_d^\Phi}$; for $b_{\Phi} < 2$,
\be
\h^\Phi (\M) = \{ x= y +z + w: y \in \h_c^\Phi (\M),\; z \in \h_r^\Phi (\M),\; w \in \h_d^\Phi (\M)\}
\ee
equipped with the norm $\| x \|_{\h^\Phi} = \inf_{x = y + z + w} \{ \| y\|_{\h_c^\Phi} + \|z\|_{\h_r^\Phi} + \|w\|_{\h_d^\Phi} \},$ where the infimum is taken over all decompositions $x= y +z +w$ with $y \in \h_c^\Phi (\M)$, $z \in \h_r^\Phi (\M)$, and $w \in \h_d^\Phi (\M)$.

First, we consider  the inequalities of noncommutative  martingale transforms. These inequalities are first proved in \cite{PX1} for  the tracial $L^p(\M)\;(1<p<\8)$ case, and then extended to the general $L^p(\M)\;(1<p<\8)$ case in \cite{JX}. For the symmetric spaces case see \cite{DPPS}.

\begin{theorem}\label{thm:martingale-transformation} There is a constant $C_\Phi$ such that for any finite noncommutative $L^\Phi$-martingale  $x=(x_{m})_{m\geq 1}$,
\be
\|\sum_{m\ge1}\varepsilon_mdx_m\|_{L^{\Phi}({\mathcal{M}})}\le C_\Phi\|x\|_{L^{\Phi}({\mathcal{M}})},\qquad \varepsilon_m=\pm1.
\ee
\end{theorem}
\begin{proof} Let $x$ be a finite martingale  in $L^\Phi(\M)$ relative to $(\M_m)_{m\ge1}$.
Lifting $x$ to $L^\Phi(\U(\M))$, we consider $x$ as a martingale relative to $(\U(M_m))_{m\ge1}$. Set
\be
x^{(n)}=\F_n(x)=(\F_n(x_m))_{m\ge1},\qquad n\in \mathbb{N}.
\ee
Then $x^{(n)}$ is a martingale in $L^\Phi(\U_n(\M))$ relative to $(\U_n(\M_m))_{m\ge1}$. Notice that $\U_n(\M)$
admits a normal faithful finite trace $\varphi_n$. Using  the construction of $\varphi_n$ in section 4  and \eqref{eq:equality-condition}, we know that
$\varphi_n$ is invariant under $\hat{\E}_m$ for all $m\in \mathbb{N}$. Since  the martingale
structure determined by $(\U_n(\M_m))_{m\ge1}$ in $L^\Phi(\U_n(\M))$ coincides with that in the
tracial noncommutative $L^\Phi$-space of $\U_n(\M)$,
by \cite[Proposition 4.16]{DPPS}, we have that
\be
\|\sum_{m\ge1}\varepsilon_mdx_m^{(n)}\|_{L^\Phi(\U_n(\M))}\le C_\Phi\|x^{(n)}\|_{L^\Phi(\U_n(\M))},\qquad \varepsilon_m=\pm1.
\ee
Set $z=\sum_{m\ge1}\varepsilon_mdx_m$. Applying \eqref{eq:condtional expectations-conditional expectations}, we get that $\F_n(z)=\sum_{m\ge1}\varepsilon_mdx_m^{(n)}$.
On the other hand, $(\F_n(z))_{n\ge1}$ is a martingale in $L^\Phi(\U(\M))$ with respect to $(\U_n(\M))_{n\ge1}$.  Therefore,
$\F_n(z)\rightarrow z$ in $L^\Phi(\U(\M))$ as $n\rightarrow\8$ (see (ii) of Remark \eqref{re:conver}). Similarly, $x^{(n)}=\F_n(x)\rightarrow z$ in $L^\Phi(\U(\M))$ as $n\rightarrow\8$. Hence,
\be
\|\sum_{m\ge1}\varepsilon_mdx_m\|_{L^{\Phi}({\mathcal{M}})}\le C_\Phi\|x\|_{L^{\Phi}({\mathcal{M}})}.
\ee
\end{proof}

Second, we will extend the noncommutative Stein inequality to the $L^\Phi(\M)$ case.  For the tracial $L^p$-space  case see  \cite{PX1}, for the Haagerup $L^p$-space case see Junge/Xu \cite{JX} and for the symmetric space see \cite{B}.

\begin{lemma}\label{lem:stein-convergence}  Let $a=(a_m)_{m\ge1}\in L^\Phi(\M;\ell_2)$. Set
\be
a^{(n)}=\F_n(a)=(\F_n(a_m))_{m\ge1},\qquad n\in \mathbb{N}.
\ee
Then $a^{(n)}\rightarrow a$ in $L^\Phi(\U(\M);\ell_2)$ as $n\rightarrow\8$. A similar statement holds for the row
space.
\end{lemma}
\begin{proof}
Since $L^{\Phi}(\M, \el_{c}^{2})$ ( respectively, $L^{\Phi}(\U(\M), \el_{c}^{2})$)  can  be viewed
as a subspace of $L^{\Phi}({\mathcal{M}}\otimes {\mathcal{B}}(\ell_{2}))$ ( respectively, $L^{\Phi}(\U(\M)\otimes {\mathcal{B}}(\ell_{2}))$ ) as column  vectors. It is clear that
\be
a^{(n)}=\F_n\otimes id_{{\mathcal{B}}(\ell_{2})}(a).
\ee
On the other hand, $(\F_n)_{n\ge1}$ is an increasing sequence of conditional expectations. Using the martingale mean convergence in (ii) of Remark \ref{re:conver}, we obtain the desired result.
\end{proof}

\begin{theorem}\label{thm:stein inequality} There is a constant $\alpha_\Phi$ such that for any finite sequence $(a_m)_m$ in  $L^\Phi(\M)$,
\be
\|(\sum_{m\ge1}|\E_m(a_m)|^2)^\frac{1}{2}\|_{L^{\Phi}({\mathcal{M}})}\le \alpha_\Phi\|(\sum_{m\ge1}|a_m|^2)^\frac{1}{2}\|_{L^{\Phi}({\mathcal{M}})}.
\ee
Consequently, $\H_c^\Phi(\M)$ and $\H_r^\Phi(\M)$ are complemented in $L^{\Phi}(\M, \el_{c}^{2})$ and $L^{\Phi}(\M, \el_{r}^{2})$, respectively.
\end{theorem}
\begin{proof} We use the same method as in the proof  of Theorem \ref{thm:martingale-transformation}.
Fix a finite sequence $(a_m)_m\subset L^\Phi(\M)\subset L^\Phi(\U(\M))$. Then $(\F_n(a_m))_m\subset L^\Phi(\U_n(\M))$. Using \cite[Lemma 2.2]{B}, we get
\be
\|(\E_m(\F_n(a_m)))_{m\ge1}\|_{L^{\Phi}(\U_n(\M), \el_{c}^{2}))}\le \alpha_\Phi\|(\F_n(a_m))_{m\ge1}\|_{L^{\Phi}(\U_n(\M), \el_{c}^{2})}.
\ee
Since $\E_m(\F_n(a_m))=\F_n\E_m((a_m))$ for all $n,\;m\in \mathbb{N}$ (see \eqref{eq:condtional expectations-conditional expectations}),
\be
\|(\F_n(\E_m(a_m)))_{m\ge1}\|_{L^{\Phi}(\U_n(\M), \el_{c}^{2}))}\le \alpha_\Phi\|(\F_n(a_m))_{m\ge1}\|_{L^{\Phi}(\U_n(\M), \el_{c}^{2})}.
\ee
By Lemma \ref{lem:stein-convergence}, we obtain the desired result.
\end{proof}

Third, we consider Burkholder-Gundy inequalities of noncommutative martingales. Pisier/Xu proved these inequalities in \cite{PX1} for the tracial $L^p$-space  case  and Junge/Xu \cite{JX} proved  for the Haagerup $L^p$-space case $(1<p<\8)$. These inequalities have extended for the symmetric space and $\Phi$-moment case (see \cite{BC,BCLJ,D}).

\begin{lemma}\label{lem:Burkholder-Gundy-complemented} Let $\mathcal{H}^{\Phi}(\U(\M))$ and $\mathcal{H}^{\Phi}(\U_n(\M))\;(n\in \mathbb{N})$ be the Hardy spaces on $\U(\M)$ and $\U_n(\M)$
 relative to the filtrations $(\U(\M_m))_{m\ge1}$ and $(\U_n(\M_m))_{m\ge1}$, respectively. Then  $\mathcal{H}^{\Phi}(\M)$ and $\mathcal{H}^{\Phi}(\U_n(\M))$ are complemented
isomorphic subspaces of  $\mathcal{H}^{\Phi}(\U(\M))$.
\end{lemma}
\begin{proof} It is clear that $\H_c^\Phi(\M)$ (respectively, $\H_r^\Phi(\M)$ ) is an isometric
subspace of $\H_c^\Phi(\U(\M))$  (respectively, $\H_r^\Phi(\U(\M))$ ). Let  $\hat{x}=(\hat{x}_{m})_{m\geq 1}$ be a martingale relative
to $(\U(\M_m))_{m\ge1}$. Set $\F(\hat{x})=(\F(\hat{x}_{m}))_{m\geq 1}$. Then $\F(\hat{x})$ is a martingale relative to
$(\M_m)_{m\ge1}$ (see \eqref{eq:condtional expectations-conditional expectations}). Note that the difference sequence of $\F(\hat{x})$  is equal to $(\F(d\hat{x}_{m}))_{m\geq 1}$.
On the other hand, we regard $L^{\Phi}(\M, \el_{c}^{2})$ ( respectively, $L^{\Phi}(\U(\M), \el_{c}^{2})$) as the column subspace of $L^{\Phi}({\mathcal{M}}\otimes {\mathcal{B}}(\ell_{2}))$ ( respectively, $L^{\Phi}(\U(\M)\otimes {\mathcal{B}}(\ell_{2}))$ ) and
\be
\F(d\hat{x}_{m}))_{m\geq 1}=\F\otimes id_{{\mathcal{B}}(\ell_{2})}(d\hat{x}).
\ee
It follows that the map $\hat{x}\mapsto\F(\hat{x})$ gives a projection from  $\H_c^\Phi(\U(\M))$  (respectively,  $\H_r^\Phi(\U(\M))$ ) onto $\H_c^\Phi(\M)$ (respectively,  $\H_r^\Phi(\M)$ ). Hence,  $\mathcal{H}^{\Phi}(\M)$ is a complemented
isomorphic subspaces of  $\mathcal{H}^{\Phi}(\U(\M))$. Similarly, by the same arguments, we get $\mathcal{H}^{\Phi}(\U_n(\M))$ is a complemented
isomorphic subspaces of  $\mathcal{H}^{\Phi}(\U(\M))$.
\end{proof}

\begin{lemma}\label{lem:Burkholder-Gundy-convergence} Let  $x\in\mathcal{H}^{\Phi}(\M)$. Then   $\F_n(x)\in\mathcal{H}^{\Phi}(\U(\M))$ for all $n\in \mathbb{N}$ and $\F_n(x)\rightarrow x$ in $\mathcal{H}^{\Phi}(\U(\M))$ as $n\rightarrow\8$.
\end{lemma}
\begin{proof} From the proof of Lemma \ref{lem:stein-convergence}, we know that if $x\in \H_c^\Phi(\M)$ (respectively, $x\in \H_r^\Phi(\M)$ ), then   $\F_n(x)\in \H_c^\Phi(\U(\M))$  (respectively, $\F_n(x)\in \H_r^\Phi(\U(\M))$ ) for all $n\in \mathbb{N}$ and $\F_n(x)\rightarrow x$ in $\H_c^\Phi(\U(\M))$  (respectively, $\H_r^\Phi(\U(\M))$ ) as $n\rightarrow\8$. In the case $2\leq a_{\Phi}$,  the result immediately holds.  Since the set of all finite $L^{\Phi}$-martingales is dense in $\mathcal{H}^{\Phi}(\M)$, the result also holds for the case  $b_{\Phi}<2$.
\end{proof}
\begin{theorem}\label{thm:Burkholder-Gundy}  There are  constants $A_\Phi,\;B_\Phi$ such that for
any finite noncommutative $L^\Phi$-martingale  $x=(x_{m})_{m\geq 1}$,
\be
A_\Phi\|x\|_{\mathcal{H}^{\Phi}}\le\|x\|_{L^\Phi(\M)}\le B_\Phi\|x\|_{\mathcal{H}^{\Phi}}
\ee
\end{theorem}
\begin{proof}
The proof is similar to that of Theorem \ref{thm:stein inequality} with the help of Lemma \ref{lem:Burkholder-Gundy-complemented} and \ref{lem:Burkholder-Gundy-convergence}.
\end{proof}

Fourth, we prove Burkholder inequalities in the Haagerup noncommutative Orlicz space case.  Junge/Xu \cite{JX} proved these inequalities for the Haagerup $L^p$-space case $(1<p<\8)$. These inequalities have extended for the symmetric space and $\Phi$-moment case (for more information see \cite{B1,D,RW,RW1,RWX}).

Since  $L^{\Phi}_{cond}(\M, \el_{c}^{2})$ ( respectively, $L^{\Phi}_{cond}(\M, \el_{r}^{2})$)  can  be viewed
as a subspace of $L^{\Phi}({\mathcal{M}}\otimes {\mathcal{B}}(\ell_{2}(\mathbb{N}^2)))$ as column (respectively, row) vectors, using  the argument in the proof of Lemma \ref{lem:Burkholder-Gundy-complemented} and \ref{lem:Burkholder-Gundy-convergence}, we obtain the following result.
\begin{lemma}\label{lem:Burkholder}
\begin{enumerate}[\rm(i)]
 \item $\h^{\Phi}(\M)$ and $\h^{\Phi}(\U_n(\M))$ are complemented
isomorphic subspaces of  $\h^{\Phi}(\U(\M))$.
 \item If $x\in\mathcal{H}^{\Phi}(\M)$, then   $\F_n(x)\in\h^{\Phi}(\U(\M))$ for all $n\in \mathbb{N}$ and $\F_n(x)\rightarrow x$ in $\h^{\Phi}(\U(\M))$ as $n\rightarrow\8$.
 \end{enumerate}
\end{lemma}

Now, we can get the analogues of Burkholder inequality in the setting of the Haagerup noncommutative Orlicz space by arguments similar to the proof of Theorems \ref{thm:Burkholder-Gundy}.
\begin{theorem}\label{thm:Burkholder}  There are  constants $\alpha_\Phi,\;\beta_\Phi$ such that for
any finite noncommutative $L^\Phi$-martingale  $x=(x_{m})_{m\geq 1}$,
\be
\alpha_\Phi\|x\|_{\h^{\Phi}}\le\|x\|_{L^\Phi(\M)}\le \beta_\Phi\|x\|_{\h^{\Phi}}.
\ee
\end{theorem}

\begin{lemma}\label{lem:stein-conditional} For any finite sequence $a=(a_m)_{m\ge1}$ in $L^{\Phi}(\M)$
define
\be
Q(a)=(\E_m(a_m))_{m\ge1}\qquad \mbox{and}\qquad R(a)=(\E_{m-1}(a_m))_{m\ge1} \qquad(\E_0=\E_1).
\ee
Then $Q$ and $R$
 extend to  projections on $L^{\Phi}_{cond}(\M, \el_{c}^{2})$  and $L^{\Phi}_{cond}(\M, \el_{r}^{2})$. Consequently, $\h_c^\Phi(\M)$ and $h_r^\Phi(\M)$ are respectively complemented in $L^{\Phi}_{cond}(\M, \el_{c}^{2})$  and $L^{\Phi}_{cond}(\M, \el_{r}^{2})$.
\end{lemma}
\begin{proof}
Let $\N,\;\N_m$ and $\check{\E}_m$ as in \eqref{eq:condtional expectations-conditional expectations} and \eqref{eq:condiional-cross-product}. If $a=(a_m)_{m\ge1}$ is a finite sequence in $L^{\Phi}(\M)$, then $a=(a_m)_{m\ge1}\subset L^{\Phi,\8}_0(\N)$. Hence,
\be
|\check{\E}_m(a_m)|^2 = (\check{\E}_m(a_m))^*\check{\E}_m(a_m)\le \check{\E}_m(|a_m|^2)
\ee
and
\be
\check{\E}_{m-1}(|\check{\E}_m(a_m)|^2)\le\check{\E}_{m-1}(|a_m|^2).
\ee
Since $\check{\E}_m|_{L^{\Phi}(\M)}=\E_m\;( m\in \mathbb{N})$,
\be
\E_{m-1}(|\E_m(a_m)|^2)\le\E_{m-1}(|a_m|^2).
\ee
It follows that $\|Q(a)\|_{L^{\Phi}_{cond}(\M, \el_{c}^{2})}\le\|Q(a)\|_{L^{\Phi}_{cond}(\M, \el_{c}^{2})}$.

Therefore, $Q$ extends to a contractive projection on $L^{\Phi}_{cond}(\M, \el_{c}^{2})$. Similarly,
$R$ extends to a contractive projection on $L^{\Phi}_{cond}(\M, \el_{c}^{2})$.
\end{proof}

From Theorem \ref{thm:stein inequality} and Lemma \ref{lem:stein-conditional},  we obtain the following result.

\begin{proposition}\label{complement}
\begin{enumerate}[\rm(i)]
\item $\H_c^\Phi(\M),\;\H_r^\Phi(\M)$ are complemented in $L^\Phi({\mathcal{M}}\otimes {\mathcal{B}}(\ell_{2}))$.
\item $\h_c^\Phi(\M),\;\h_r^\Phi(\M)$ are complemented in $L^\Phi({\mathcal{M}}\otimes {\mathcal{B}}(\ell_{2}(\mathbb{N}^2)))$.
\end{enumerate}
\end{proposition}

By Theorem \ref{thm:dual-hagerup} and Proposition \ref{complement}, it follows that

\begin{theorem}\label{thm:maringale-duality}
\begin{enumerate}[\rm(i)]
\item $(\H_c^\Phi(\M))^{*}=\H_c^\Psi(\M) $ and $(\H_r^\Phi(\M))^{*}=\H_r^\Psi(\M) $ with equivalent norms.

\item $(\h_c^\Phi(\M))^{*}=\h_c^\Psi(\M)$ and $(\h_r^\Phi(\M))^{*}=\h_r^\Psi(\M) $ with equivalent norms.
\end{enumerate}
\end{theorem}

Using Theorem \ref{thm:dual-hagerup}, Lemma \ref{lem:stein-conditional} and the argument in the proof of \cite[Lemma 2.3]{B1}, we get that

\begin{proposition}\label{pro:duality}
 $(\h_d^\Phi(\M))^*=\h_d^\Psi(\M)$ with equivalent norms.
\end{proposition}

\begin{theorem}\label{thm:hardy-condition hardy}
 \begin{enumerate}[\rm(i)]
 \item  If $p_\Phi>2$, then $\H_c^\Phi(\M)=\h^\Phi_c(\M)\cap \h_d^\Phi(\M) $ and $\H_r^\Phi(\M)=\h^\Phi_r(\M)\cap \h_d^\Phi(\M)$ with equivalent norms.
 \item If $q_\Phi<2$, then then $\H_c^\Phi(\M)=\h^\Phi_c(\M)+\h_d^\Phi(\M) $ and $\H_r^\Phi(\M)=\h^\Phi_r(\M)+\h_d^\Phi(\M)$ with equivalent norms.
 \end{enumerate}
\end{theorem}
\begin{proof} (i) The proof is the same as that of  Theorem \ref{thm:stein inequality} through using   \cite[Proposition 2.4]{B1} and Lemma \ref{lem:Burkholder-Gundy-complemented}, \ref{lem:Burkholder-Gundy-convergence} and \ref{lem:Burkholder}.

(ii) Using (i), Theorem \ref{thm:maringale-duality} and Proposition \ref{pro:duality}, we obtain the desired result.
\end{proof}

\subsection*{Acknowledgement}
This research was funded by the Science Committee of the Ministry of Science and High
Education of the Republic of Kazakhstan (Grant No. AP14870431).

\end{document}